\documentclass[a4paper,12pt,twoside]{article}

\RequirePackage[%
   a4paper%
  ,hcentering%
  ,textwidth=16cm%
  ,textheight=25.0cm%
  ,headheight=14.5pt%
  ,footskip=1.5cm%
  ,marginparwidth=.7in%
]{geometry}

\usepackage{hyperref}

\usepackage{graphicx}
\usepackage{amssymb,amsmath}
\usepackage{amsthm}

\usepackage{fancyhdr}
\newtheorem{theorem}{Theorem}
\newtheorem{lemma}{Lemma}
\newtheorem{corollary}{Corollary}
\newtheorem{proposition}{Proposition}

\begin{document}

\author{Holger Heitsch \and Ren\'e Henrion}
\date{\small Weierstrass Institute for Applied Analysis and Stochastics\\
        Mohrenstra{\ss}e 39, 10117 Berlin, Germany\\[2ex] December 18, 2020}
\title{\bf An enumerative formula for the spherical cap discrepancy}
\maketitle

\begin{abstract}
\noindent
The spherical cap discrepancy is a widely used measure for how uniformly a sample of points on the sphere is distributed.
Being hard to compute, this discrepancy measure is typically replaced by some lower or upper estimates when designing 
optimal sampling schemes for the uniform distribution on the sphere. In this paper, we provide a fully explicit, easy
to implement enumerative formula for the spherical cap discrepancy. Not surprisingly, this formula is of combinatorial
nature and, thus, its application is limited to spheres of small dimension and moderate sample sizes.
Nonetheless, it may serve as a useful calibrating tool for testing the efficiency of sampling schemes and its
explicit character might be useful also to establish necessary optimality conditions when minimizing the discrepancy
with respect to a sample of given size.\\[1ex]
\end{abstract}

\section{Introduction}
\thispagestyle{empty}

A discrepancy measure $\Delta (\mu,\nu)$ quantifies the deviation between two given measures $\mu$ and $\nu$. On a local scale, one may compare the 
two measures with respect to a given set $B$ to obtain the so-called local discrepancy
\[
\Delta(B;\mu,\nu):=|\mu(B)-\nu(B)|.
\]
In order to arrive at a global deviation measure, one extends the comparison of the two measures to a collection $\mathcal{B}$ of sets and chooses an appropriate $L_p$ norm:
\begin{equation}\label{discgen}
\begin{aligned}
\Delta_p (\mu,\nu) &:= \left(\int_\mathcal{B}\Delta(B;\mu,\nu)^pd\omega (B)\right)^{1/p}\quad(p<\infty),\\
\Delta_\infty (\mu,\nu) &:=\,\sup_{B\in\mathcal{B}}\Delta(B;\mu,\nu)\,.\end{aligned}
\end{equation}
For surveys on discrepancies, we refer to, e.g., \cite{chen,drmota,mat}. Discrepancies play a fundamental role in many mathematical disciplines. For instance, in stochastic programming, the stability of optimal solutions and optimal values with respect to perturbations of the underlying probability measure can be expected only for a problem-adapted choice of a discrepancy \cite{rom}. 

The focus of the present paper will be on the so-called {\it spherical cap discrepancy}. Our interest in this quantity
comes from the algorithmic solution of optimization problems subject to probabilistic constraints. One approach here relies on the so-called {\it spheric-radial decomposition} of random vectors having elliptically symmetric distribution (e.g., Gaussian). This approach allows for a representation of the decision-dependent probability of some random inequality system as well as of its gradient as integrals with respect to the uniform distribution on a sphere \cite{ack}. Hence, for an efficient numerical approximation of these integrals by finite sums, one has to make use of low discrepancy samples for that distribution. It is well known (see, e.g. \cite[p. 991]{brauchart}), that the resulting integration error tends to zero (for samples of increasing size) if and only if the spherical cap discrepancy associated with these samples tends to zero. This special discrepancy is obtained from our general setting \eqref{discgen} by defining $p:=\infty$, $\mu$ as the uniform measure on the sphere, $\nu$ as the empirical measure induced by the sample and $\mathcal{B}$ as the collection of all closed half spaces intersected with the sphere (caps).
To be more precise, we define the closed halfspace $H(w,t)$ parameterized by $(w,t)$, its
empirical and cap measures $\mu ^{emp}\left( w,t\right) $ and $\mu
^{cap}\left( w,t\right) $, respectively, and the spherical cap discrepancy $\Delta $
associated with the sample $\left\{ x^{1},\ldots ,x^{N}\right\} $ by 
\begin{eqnarray*}
H(w,t) &\!:=\!&\left\{ x\in \mathbb{R}^{n}|\left\langle w,x\right\rangle \geq
t\right\} \quad \left( w\in \mathbb{S}^{n-1},\,t\in \left[ -1,1\right]
\right), \\
\mu ^{emp}\left( w,t\right) &\!:=\!&N^{-1}\cdot \#\left\{ i\in \left\{ 1,\ldots
,N\right\} |x^{i}\in H(w,t)\right\}, \\
\mu ^{cap}\left( w,t\right) &\!:=\!&\mu \left( \mathbb{S}^{n-1}\cap
H(w,t)\right) \, \left( \mu =\mbox{law of uniform distribution on }%
\mathbb{S}^{n-1}\right), \\
\Delta (w,t) &\!:=\!&\left\vert \mu ^{emp}\left( w,t\right) -\mu ^{cap}\left(
w,t\right) \right\vert, \\
\Delta &\!:=\!&\sup_{w\in \mathbb{S}^{n-1},\,t\in \left[ -1,1\right] }\Delta
(w,t).
\end{eqnarray*}
The following explicit formula for the cap measure - not depending on
\thinspace $w\in \mathbb{S}^{n-1}$ - is well known

\begin{equation}
\mu ^{cap}\left( w,t\right) =C_{n}\cdot\left\{ 
\begin{array}{clc}
\displaystyle\int_{0}^{\arccos (t)}\sin ^{n-2}(\tau )d\tau , & \quad %
\mbox{if} & 0\leq t\leq 1, \\ 
\displaystyle1-\int_{0}^{\arccos (-t)}\sin ^{n-2}(\tau )d\tau , & \quad %
\mbox{if} & -1\leq t<0,%
\end{array}%
\right.  \label{capmeasure}
\end{equation}
where 
\begin{equation*}
C_{n}:=\frac{1}{\int_{0}^{\pi }\sin ^{n-2}(\tau )d\tau }
\end{equation*}
is the normalizing constant.
\begin{figure}[htb]
\center\includegraphics[width=13cm]{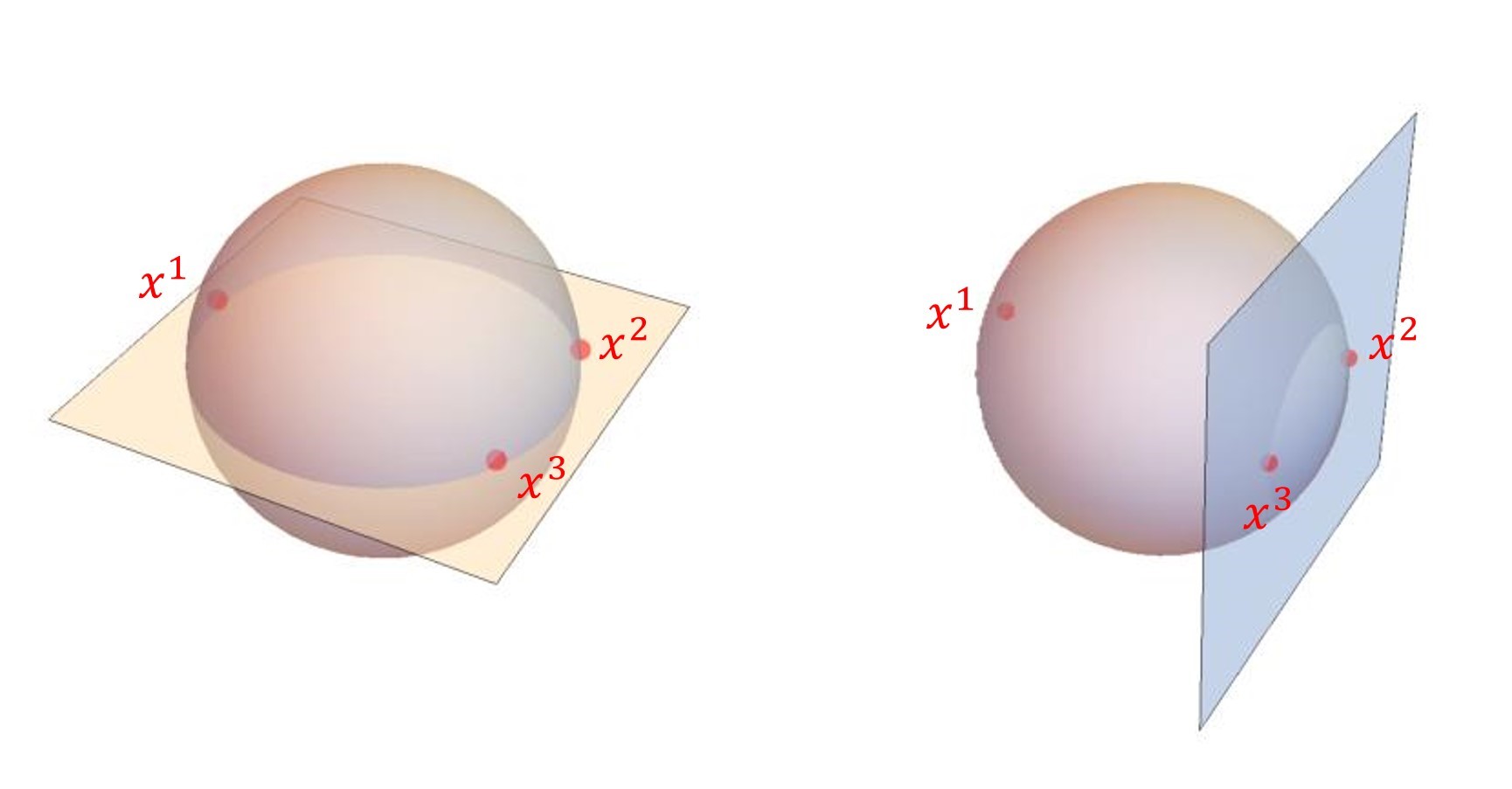}
\caption{\label{counter}Examples for spherical caps}
\end{figure}

\noindent
Figure \ref{counter} illustrates different spherical caps on $\mathbb{S}^2$  for a set of three points located in the x/y plane. This plane itself induces an upper and a lower closed halfspace each of them containing all three points (left picture). Hence, the associated upper and lower caps both have empirical measure 1 and cap measure 1/2. Therefore, the local discrepancies $\Delta (w,t)$ of these caps equal 1/2. Another hyperplane passes just through two of the three points (right picture) and the associated left and right halfspaces induce a big and a small cap. The measure of the small cap tends to zero when the two points converge to $-x_1$. Therefore, the local discrepancy related with this small cap tends to 2/3. 

To the best of our knowledge, no explicit formula for calculating the spherical cap discrepancy has been known so far. Rather the emphasis in the literature has been laid on suitable estimates with respect to more manageable quantities allowing for asymptotic derivations and constructions of efficient low discrepancy designs (see, e.g., \cite{brauchart,grab}). On the other hand, beyond the asymptotic `large sample' viewpoint it might be of some interest even for fixed moderate sample sizes to establish an easy enumerative formula enabling one to precisely compute the discrepancy and to compare different sampling schemes.

As a rule, $L_p$ discrepancies ($p<\infty$) are easier to compute than $L_\infty$ discrepancies as a consequence of the collection $\mathcal{B}$ of test sets typically having infinite cardinality \cite{doerr}. As far as explicit formulae for $L_\infty$ discrepancies are available (e.g., for rectangular or general polyhedral sets, see \cite{doerr,hen1,hen2}), they are of combinatorial nature which limits their application with respect to the dimension and the size of the sample. 
More precisely, it has been shown in \cite{gnewuch}, that computing the star discrepancy is an NP-hard problem. Moreover, the result is improved by \cite{gianno} who proved that it is indeed W[1]-hard.
Therefore, it is not surprising that a similar combinatorial aspect shows up in the enumerative formula for the spherical cap discrepancy we present in Theorem \ref{maintheo} below. In our numerical experiments, we apply the formula to spheres of dimension starting from 2 (2000 samples) up to 5 (100 samples). Even in this rather modest setting, the formula may prove useful for calibration purposes with respect to some given sampling scheme. For instance, in \cite[p. 1005]{brauchart} an easy to compute lower bound for the spherical cap discrepancy  is used in numerical experiments in order to confirm empirically a certain asymptotic order for a digital net based on a two-dimensional Sobol' point set on $\mathbb{S}^2$. Strictly speaking, the order obtained with respect to the lower bound transfers to the discrepancy only when the ratio between the true value and the lower estimate is approximately constant for increasing sample size. This is what we may confirm indeed in our numerical experiments.
We also use the proven formula in order to directly compare discrepancies of a few sampling schemes on $\mathbb{S}^2$ for sample sizes of up to 1000. The results verify the good quality of a sampling scheme via Lambert's equal-area transform proposed in \cite[p. 995]{brauchart}. 
Finally, we mention that the explicit character of the obtained formula might be of some interest for the derivation of necessary optimality conditions when minimizing the discrepancy as a function of a sample of fixed size.

\section{Preparatory Results}

We have the following elementary (semi-) continuity properties of both
considered measures:
\begin{lemma}
\label{ohs}$\mu ^{cap}$ is continuous and\ $\mu ^{emp}$ is upper
semicontinuous on $\mathbb{S}^{n-1}\times \left[ -1,1\right] $. Moreover,
the following relations are satisfied for all $w\in \mathbb{S}^{n-1}$ and $%
\,t\in \left[ -1,1\right] $: 
\begin{eqnarray}
\mu ^{emp}\left( w,t\right) +\mu ^{emp}\left( -w,-t\right) \nonumber&=&\\
1+N^{-1}\#\{i|x^i\in H(w,t)\cap H(-w,-t)\} &\geq &1
\label{spiegel}, \\
\mu ^{cap}\left( w,t\right) +\mu ^{cap}\left( -w,-t\right) &=&1.
\label{spiegel2}
\end{eqnarray}
\end{lemma}

\begin{proof}
The continuity of $\mu ^{cap}$ and (\ref{spiegel2}) follow immediately from 
(\ref{capmeasure}), while (\ref{spiegel}) is an immediate consequence of the definitions.
Let $w\in \mathbb{S}^{n-1},\,t\in 
\left[ -1,1\right] $ and $\left( w_{k},t_{k}\right) \rightarrow \left(
w,t\right) $ an arbitrary sequence with $w_{k}\in \mathbb{S}%
^{n-1},\,t_{k}\in \left[ -1,1\right] $. Define 
\begin{equation*}
I:=\left\{ i\in \left\{ 1,\ldots ,N\right\} |x^{i}\notin H(w,t)\right\} ,
\end{equation*}
so that $\left\langle w,x^{i}\right\rangle <t$ for all $i\in I$. Then, by
continuity, there is some $k_{0}$, such that $\left\langle
w_{k},x^{i}\right\rangle <t_{k}$ - i.e., $x^{i}\notin H(w_{k},t_{k})$ - for
all $k\geq k_{0}$ and all $i\in I$. It follows that 
\begin{equation*}
\mu ^{emp}\left( w_{k},t_{k}\right) \leq \mu ^{emp}\left( w,t\right) \quad
\forall k\geq k_{0},
\end{equation*}
whence 
\begin{equation*}
\limsup_{k\rightarrow \infty }\mu ^{emp}\left( w_{k},t_{k}\right) \leq \mu
^{emp}\left( w,t\right) .
\end{equation*}
This proves the upper semicontinuity of $\mu ^{emp}$ on $\mathbb{S}%
^{n-1}\times \left[ -1,1\right] $.   
\end{proof}

Figure \ref{discrepplot} plots the local discrepancy $\Delta (w,t)$ for the unit circle $\mathbb{S}^1$. As can be seen, it is a highly irregular, discontinuous (actually neither upper nor lower semicontinuous yet piecewise smooth) function. Therefore it is not à priori evident that the supremum in the definition of the discrepancy is attained.
\begin{figure}[htb]
\center\includegraphics[width=12cm]{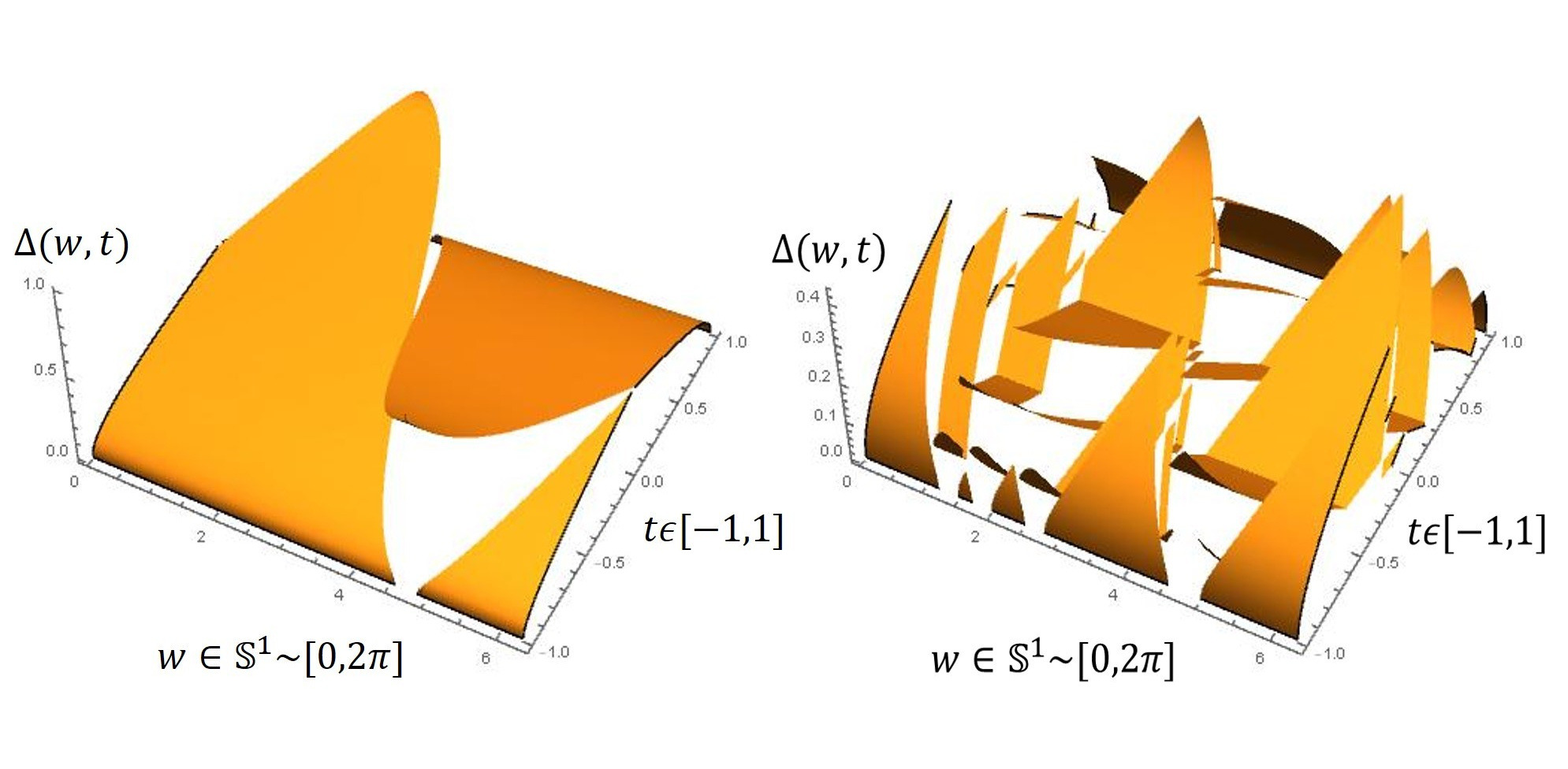}
\caption{\label{discrepplot}Plot of $\Delta (w,t)$ for a sample of size $N=1$ (left) and $N=5$ (right) on $\mathbb{S}^1$.}
\end{figure}

\noindent
The next proposition shows that the discrepancy $\Delta $ is
always realized indeed by a certain cap:

\begin{proposition}
\label{maxsup}There are $w^{\ast }\in \mathbb{S}^{n-1},\,t^{\ast }\in \left[
-1,1\right] $, such that 
\begin{equation*}
\Delta =\left\vert \mu ^{emp}\left( w^{\ast },t^{\ast }\right) -\mu
^{cap}\left( w^{\ast },t^{\ast }\right) \right\vert .
\end{equation*}
\end{proposition}

\begin{proof}
Let $\left( w_{k},t_{k}\right) \in \mathbb{S}^{n-1}\times \left[ -1,1\right] 
$ be a sequence realizing the supremum in the definition of $\Delta $: 
\begin{equation}
\left\vert \mu ^{emp}\left( w_{k},t_{k}\right) -\mu ^{cap}\left(
w_{k},t_{k}\right) \right\vert \rightarrow _{k}\Delta  \label{supreal}
\end{equation}%
By the compactness of $\mathbb{S}^{n-1}\times \left[ -1,1\right] $ we may
assume that 
\begin{equation}
\left( w_{k},t_{k}\right) \rightarrow \left( \bar{w},\bar{t}\right) \in 
\mathbb{S}^{n-1}\times \left[ -1,1\right] .  \label{cluster}
\end{equation}%
According to (\ref{supreal}) one may assume one of the following two cases
upon passing to a subsequence: 
\begin{eqnarray}
\mu ^{emp}\left( w_{k},t_{k}\right) -\mu ^{cap}\left( w_{k},t_{k}\right)
&\rightarrow &\Delta  \label{fall1}, \\
\mu ^{cap}\left( w_{k},t_{k}\right) -\mu ^{emp}\left( w_{k},t_{k}\right)
&\rightarrow &\Delta .  \label{fall2}
\end{eqnarray}%
In the case of (\ref{fall1}), the continuity of $\mu ^{cap}$ and the upper
semicontinuity of $\mu ^{emp}$ on $\mathbb{S}^{n-1}\times \left[ -1,1\right] 
$ (see Lemma \ref{ohs}) yield along with (\ref{cluster}) that:%
\begin{eqnarray*}
\Delta &=&\lim_{k\rightarrow \infty }\left( \mu ^{emp}\left(
w_{k},t_{k}\right) -\mu ^{cap}\left( w_{k},t_{k}\right) \right)\\
&=&\limsup_{k\rightarrow \infty }\left( \mu ^{emp}\left( w_{k},t_{k}\right)
-\mu ^{cap}\left( w_{k},t_{k}\right) \right) \\
&\leq &\mu ^{emp}\left( \bar{w},\bar{t}\right) -\mu ^{cap}\left( \bar{w},%
\bar{t}\right) \leq \left\vert \mu ^{emp}\left( \bar{w},\bar{t}\right) -\mu
^{cap}\left( \bar{w},\bar{t}\right) \right\vert \leq \Delta .
\end{eqnarray*}
Hence, $\Delta =\left\vert \mu ^{emp}\left( \bar{w},\bar{t}\right) -
\mu^{cap}\left( \bar{w},\bar{t}\right) \right\vert $. In the case of (\ref%
{fall2}) one may exploit (\ref{spiegel}), (\ref{spiegel2}) and once more the
upper semicontinuity of \,$\mu ^{emp}$\, in order to derive that: 
\begin{eqnarray*}
\Delta &=&\lim_{k\rightarrow \infty }\left( \mu ^{cap}\left(
w_{k},t_{k}\right) -\mu ^{emp}\left( w_{k},t_{k}\right) \right)\\
&=&\lim_{k\rightarrow \infty }\left( 1-\mu ^{emp}\left( w_{k},t_{k}\right)
-(1-\mu ^{cap}\left( w_{k},t_{k}\right) \right) \\
&=&\limsup_{k\rightarrow \infty }\left( 1-\mu ^{emp}\left(
w_{k},t_{k}\right) -(1-\mu ^{cap}\left( w_{k},t_{k}\right) \right) \\
&\leq &\limsup_{k\rightarrow \infty }\left( \mu ^{emp}\left(
-w_{k},-t_{k}\right) -(1-\mu ^{cap}\left( w_{k},t_{k}\right) \right) \\
&\leq &\mu ^{emp}\left( -\bar{w},-\bar{t}\right) -(1-\mu ^{cap}\left( \bar{w}%
,\bar{t}\right) )=\mu ^{emp}\left( -\bar{w},-\bar{t}\right) -\mu
^{cap}\left( -\bar{w},-\bar{t}\right) \\
&\leq &\left\vert \mu ^{emp}\left( -\bar{w},-\bar{t}\right) -\mu
^{cap}\left( -\bar{w},-\bar{t}\right) \right\vert \leq \Delta .
\end{eqnarray*}
Hence, $\Delta =\left\vert \mu ^{emp}\left( -\bar{w},-\bar{t}\right) -\mu
^{cap}\left( -\bar{w},-\bar{t}\right) \right\vert $. Altogether, the
assertion follows with 
\linebreak
$\left( w^{\ast },t^{\ast }\right)
:=\left( \bar{w},\bar{t}\right) $ in the first case and
$\left( w^{\ast },t^{\ast }\right):=\left( -\bar{w},-\bar{t}\right) $
in the second one.  
\end{proof}

We may strengthen the previous Proposition in the sense that not only there exists some cap realizing the discrepancy but
that it also has to contain at least one sample point on its relative boundary:
\begin{proposition}\label{nonvoid}
For $(w^*,t^*)$ realizing the discrepancy in Proposition \ref{maxsup} it holds that there is some $i\in\{1,\ldots,N\}$ such that $\langle w^*,x^i\rangle=t^*$.
\end{proposition}
\begin{proof}
Assume that $\langle w^*,x^j\rangle\neq t^*$ for all $j\in\{1,\ldots,N\}$. Then,
\begin{equation}
\mu^{emp}(w^*,t)=\mu^{emp}(w^*,t^*)\label{aux2}
\end{equation}
for $t$ close to $t^*$.
If $|t^*|<1$, then one may strictly increase ($t>t^*$) or decrease ($t<t^*$) $\mu^{cap}(w^*,t)$, so that by virtue of \eqref{aux2} the local discrepancy
$\Delta (w^*,t)$ can be strictly increased in comparison with the maximal one $\Delta (w^*,t^*)=\Delta$. This is a contradiction. If $t^*=1$, then
\begin{eqnarray*}
&\langle w^*,x^j\rangle < 1\quad\forall j\in\{1,\ldots,N\},&\label{aux1}\\
&\mu^{cap}(w^*,t^*) =\mu^{emp}(w^*,t^*)=0.&
\end{eqnarray*}
Since $\mu^{cap}(w^*,t)$ is strictly increased for $t<t^*=1$ while $\mu^{emp}(w^*,t)=0$ for $t$ close to $t^*$ (see \eqref{aux2}),
one may strictly increase the local discrepancy again, so that the same contradiction results. The case $t^*=-1$ follows analogously.  
\end{proof}

An interesting consequence of the previous Proposition is that a cap realizing the discrepancy has always empirical measure not smaller than cap measure:

\begin{corollary}\label{corollary}
For $(w^*,t^*)$ realizing the discrepancy in Proposition \ref{maxsup} it holds that 
\linebreak
$\mu^{emp}(w^*,t^*)\geq \mu^{cap}(w^*,t^*)$.
\end{corollary}
\begin{proof}
Suppose to the contrary, that $\mu^{emp}(w^*,t^*)< \mu^{cap}(w^*,t^*)$. Then, using (\ref{spiegel}) and (\ref{spiegel2}), we arrive at the contradiction
\begin{eqnarray*}
\Delta&=&\mu^{cap}(w^*,t^*)-\mu^{emp}(w^*,t^*)\\
&=&1-\mu^{cap}(-w^*,-t^*)-\\&&(1+N^{-1}\#\{i|x^i\in H(w^*,t^*)\cap H(-w^*,-t^*)\}-\mu^{emp}(-w^*,-t^*))\\
&<&\mu^{emp}(-w^*,-t^*)-\mu^{cap}(-w^*,-t^*)\leq\Delta .
\end{eqnarray*}
Here, the strict inequality follows from the fact that $H(w^*,t^*)\cap H(-w^*,-t^*)$ contains at least one sample point by Proposition \ref{nonvoid}.  
\end{proof}

\begin{lemma}
\label{setopt} Let $\left\{ x^{1},\ldots ,x^{k}\right\} \subseteq \mathbb{S}%
^{n-1}$ for some $k\in \mathbb{N}$ and let 
\begin{equation}
S := \left\{(w,t)\,\left|\,\left\langle w,x^{i}\right\rangle =t\,\,\left(
i=1,\ldots ,k\right)\right.\right\} \label{opt0}
\end{equation}
Let
\begin{equation*}
p:=\mathrm{rank\,}\left\{ \left({ x^{i} }\atop {-1 }\right) \right\}
_{i=1,\ldots ,k}.
\end{equation*}
Then, assuming without loss of generality that 
\begin{equation*}
\mathrm{rank\,}\left\{ \left({ x^{i} }\atop{-1 }\right) \right\}
_{i=1,\ldots ,p}=p
\end{equation*}
(i.e., the first $p$ points $x^i$ are affinely independent), the set $S$ defined in (\ref{opt0}) has a reduced representation 
\begin{equation}
S = \left\{(w,t)\,\left|\,\left\langle w,x^{i}\right\rangle =t\,\,\left(
i=1,\ldots ,p\right)\right.\right\}.
\label{opt01}
\end{equation}
\end{lemma}

\begin{proof}
By $p\leq k$ it is sufficient to show that the right-hand side of (\ref%
{opt01}) is contained in $S$ as defined in (\ref{opt0}). It is
therefore enough to show the implication%
\begin{equation}
\left\langle w,x^{j}\right\rangle =t\,\,\left( j=1,\ldots ,p\right)
\Longrightarrow \left\langle w,x^{i}\right\rangle =t\,\,\left( i=p+1,\ldots
,k\right).  \label{impli}
\end{equation}
By definition of $p$, the vectors $\displaystyle \left({x^{i}}\atop {-1}%
\right)$ $\left( i=p+1,\ldots ,k\right) $ are linear combinations of the
vectors $\displaystyle \left({x^{j}}\atop {-1}\right)$ $\left( j=1,\ldots
,p\right) $. Hence, for an arbitrarily fixed \,$i\in\{p+1,\ldots ,k\}$\,
there exists some $\lambda \in \mathbb{R}^{p}$ such that 
\begin{equation*}
\left({ x^{i} }\atop {-1}\right) =\sum_{j=1}^{p}\lambda _{j}\left({
x^{j} }\atop {-1}\right) .
\end{equation*}
Along with the assumption in (\ref{impli}), both components of this last
identity yield that 
\begin{equation*}
\,\left\langle w,x^{i}\right\rangle =\sum_{j=1}^{p}\lambda _{j}\left\langle
w,x^{j}\right\rangle =t\sum_{j=1}^{p}\lambda _{j}=t
\end{equation*}
which is the conclusion of (\ref{impli}).  
\end{proof}

The proof of Proposition \ref{nonvoid} might suggest the idea that a discrepancy realizing cap has to contain not just one but a maximum possible number of sample points on its relative boundary. This intuition is wrong as can be seen from Figure \ref{counter}. Here, any of the two caps in the left picture contains three points on its relative boundary but realizes a strictly smaller local discrepancy $\Delta (w,t)$ than the small cap in the right picture which contains just two of the three sample points on its relative boundary. As a consequence, the evaluation of the discrepancy $\Delta$ cannot be based just on a simple enumeration of local discrepancies $\Delta (w,t)$ associated with affinely independent subsets of the sample points. One has also to consider smaller subsets of sample points for which the hyperplane associated with the cap is not yet fixed. In order to get rid of the remaining degree of freedom, one has to maximize the local discrepancy conditionally to this small subset belonging to the relative boundary of the cap. In the right picture of Figure \ref{counter}, among all caps having $x^2,x^3$ on its boundary, the one defined by a vertical hyperplane turned out to maximize the local discrepancy. The crucial argument in order to incorporate this maximization aspect, is provided in the following result:
\begin{lemma}
\label{crit}Let $\left\{ x^{1},\ldots ,x^{k}\right\} \subseteq \mathbb{S}%
^{n-1}$ be such that
\begin{equation*}
\mathrm{rank\,}\left\{ \left( {x^{i}}\atop {-1}\right) \right\} _{i=1,\ldots
,k}=k.
\end{equation*}
Denote by $X_*$ the matrix whose columns are generated by $x^{i}$ for $%
i=1,\ldots ,k$ and define 
\begin{equation*}
\tilde{X}_*:=\left( {X_*}\atop {-\mathbf{1}^{T}}\right) ;\quad \gamma :=\mathbf{1%
}^{T}\left( \tilde{X}_*^{T}\tilde{X}_*\right) ^{-1}\mathbf{1};\quad \mathbf{1}%
:=(\,1,\ldots ,1)^{T}.
\end{equation*}
Let $\left( w^{\ast },t^{\ast }\right) $ be a local solution of the
 optimization problem
\begin{equation}
\max_{w,t}\left\{ t\,\left\vert \,\left\langle
w,x^{i}\right\rangle =t\,\,\left( i=1,\ldots ,k\right) ,\,\,\left\langle
w,w\right\rangle =1\right. \right\} .  \label{opt1}
\end{equation}
Then, it holds that $0<\gamma \leq 1$. If $\gamma <1$, then 
\[
t^*\in\left\{ \pm \left(\frac{1-\gamma }{\gamma }\right) ^{1/2}\right\},\quad w^*=
\frac{1+(t^*)^2}{t^*}X_*\left( \tilde{X}_*^{T}\tilde{X}_*\right) ^{-1}\mathbf{1}.
\]
Moreover, $\gamma =1$ is equivalent to $t^{\ast }=0$ and we then have $\mathrm{rank\,}X_* = k-1$.
\end{lemma}

\begin{proof}
In order to identify $\left( w^{\ast },t^{\ast }\right) $ via necessary
optimality conditions we have first to check if the gradients 
\begin{equation*}
\left\{ \left( {{{x}^{1}}}\atop {{-1}}\right) ,\ldots ,\left( {{{x}^{k}%
}}\atop {{-1}}\right) ,\left( {{2w}}\atop {{0}}\right) \right\}
\end{equation*}%
with respect to $\left( w,t\right) $ of the equality constraints in (\ref%
{opt1}) are linearly independent. We assume a linear combination 
\begin{equation*}
\left( {{0}}\atop {{0}}\right) =\sum_{i=1}^{k}\lambda _{i}\left( {{{\ {%
x}}^{i}}}\atop {{-1}}\right) +\mu \left( {{{\ {2w}}}}\atop {{0}}\right) .
\end{equation*}
Multiplication of the first component with $w$ yields - taking into account
the equality constraints in (\ref{opt1}) and comparing the second component
- that
\begin{equation*}
0=\sum_{i=1}^{k}\lambda _{i}\left\langle w,x^{i}\right\rangle +2\mu
\left\langle w,w\right\rangle =t\sum_{i=1}^{k}\lambda _{i}+2\mu =2\mu .
\end{equation*}
Hence, 
\begin{equation*}
\sum_{i=1}^{k}\lambda _{i}\left( {{{\ {x}}^{i}}}\atop {{-1}}\right) =\left( 
{{0}}\atop {{0}}\right) .
\end{equation*}
By assumption of the Lemma, the vectors $\left\{ \binom{{{\ {x}}^{i}}}{-1}%
\right\} _{i=1,\ldots ,k}$are linearly independent, whence $\lambda _{i}=0$
for $i=1,\ldots ,k$. Furthermore, $\mu =0$, which altogether proves the
linear independence of the gradients of equality constraints in (\ref{opt1}).

This allows us to derive the following necessary optimality
condition for a local solution $\left( w^{\ast },t^{\ast }\right) $ of
problem (\ref{opt1}). Here the gradient of the objective
function $t$ appears on the left-hand side: 
\begin{equation}
\exists \,\lambda _{1},\ldots ,\lambda _{k},\mu \in \mathbb{R}:\left( {%
{0}}\atop {{1}}\right) =\sum_{i=1}^{k}\lambda _{i}\left({x^i}\atop{-1}
\right) +\mu \left( {{2}w^{\ast }}\atop{{0}}\right) .  \label{necopt}
\end{equation}
The second component implies that $\sum_{i=1}^{k}\lambda _{i}=-1$.
Multiplication of the first component by $w^{\ast }$ and exploiting the
equality constraints in (\ref{opt1}) yields that 
\begin{equation*}
0=\sum_{i=1}^{k}\lambda _{i}\left\langle w^{\ast },x^{i}\right\rangle +2\mu
\left\langle w^{\ast },w^{\ast }\right\rangle =t^{\ast
}\sum_{i=1}^{k}\lambda _{i}+2\mu ,
\end{equation*}
in particular, $t^{\ast }=2\mu $. With $\lambda :=(\,\lambda _{1},\ldots
,\lambda _{k})^{T},$ the equation in \eqref{necopt} reads as 
\begin{equation}
\left( {0}\atop {1}\right) =\tilde{X}_*\lambda +t^{\ast}\left( {w^{\ast }}\atop {0}%
\right) .  \label{necopt2}
\end{equation}
Multiplication of both sides from the left by $\tilde{X}_*^{T}$ and using the
first feasibility constraint $X_*^{T}w^{\ast }=t^{\ast }\mathbf{1}$ in (\ref%
{opt1}) results in 
\begin{equation*}
-\mathbf{1}=\tilde{X}_*^{T}\tilde{X}_*\lambda +t^{\ast 2}\mathbf{1}.
\end{equation*}
By the assumption of this Lemma, the matrix $\tilde{X}_*^{T}\tilde{X}_*$ is
regular and we can solve the last equation for $\lambda $: 
\begin{equation}
\lambda =-(1+t^{\ast 2})\left( \tilde{X}_*^{T}\tilde{X}_*\right) ^{-1}\mathbf{1}.
\label{necopt3}
\end{equation}
Recalling that $\mathbf{1}^{T}\lambda =-1$ we arrive at 
\begin{equation}
(1+t^{\ast 2})\mathbf{1}^{T}\left( \tilde{X}_*^{T}\tilde{X}_*\right) ^{-1}%
\mathbf{1}=1\,.  \label{necopt4}
\end{equation}
By definition of $\gamma $, the latter equation implies that we necessarily
have $0<\gamma \leq 1$, and $\gamma =1$ if and only if $t^{\ast }=0$. 
In the case $\gamma = 1$ we have that $w^{\ast }\in \mathrm{Ker\,}X_*^{T}\cap \mathbb{S}^{n-1}$
by feasibility of $w^*$ in (\ref{opt1}). Moreover,
from (\ref{necopt2}) we see that $X_*\lambda = 0$ for some $\lambda\neq 0$ which then implies 
that $\mathrm{rank\,}X_* = k-1$ due to 
\begin{equation*}
k - 1 = \mathrm{rank\,} \tilde{X}_* - 1 \leq \mathrm{rank\,} X_* = k - \mathrm{dim\,} \mathrm{Ker\,} X_* \leq k - 1.
\end{equation*}
If, in contrast, $0<\gamma <1$, then, by (\ref{necopt4}),
\begin{equation}
t^{\ast }=\pm \left( \frac{1-\gamma }{\gamma }\right) ^{1/2}.
\label{necopt5}
\end{equation}
The first component of \eqref{necopt2} reads 
\begin{equation*}
0=X_*\lambda +t^{\ast }w^{\ast }.
\end{equation*}
Using the representation \eqref{necopt3} for $\lambda $ we obtain that 
\begin{equation*}
w^{\ast }=\frac{1+t^{\ast 2}}{t^{\ast }}X_*\left( \tilde{X}_*^{T}\tilde{X}_*
\right) ^{-1}\mathbf{1}
\end{equation*}
which completes the proof.
\end{proof}

Note, the fact that $(w^{\ast},t^{\ast})$ is a local maximum in \eqref{opt1} does not imply $t^{\ast}\geq 0$ in case of $k=n$. 
We proceed with the following purely technical Lemma which will be needed to cope with a degenerate subcase in our main result later on. 

\begin{lemma}\label{rankselect}
Let $\left\{ x^{1},\ldots ,x^{N}\right\} \subseteq \mathbb{S}^{n-1}$. For
any $I\subseteq \{1,\ldots ,N\}$ let $X_{I}$ be the matrix whose columns are 
$x^{i}$, $i\in I$. Define $\tilde{X}_{I}:=\left( {X_{I}}\atop {-\mathbf{1%
}^{T}}\right)$ and let be $\tilde{X} := \tilde{X}_{\{1,\ldots,N\}}$.
Let $w_0\in \mathbb{S}^{n-1}$ be given such that
\[
I_0 := \left\{ i\in\{1,\ldots,N\}\,|\, \langle w_0,x^i\rangle = 0\right\} \neq \emptyset
\]
and such that it holds
\begin{eqnarray}
    &w_0 \in \underset{w\in \mathrm{Ker\,}X_{I_0}^{T}\cap \mathbb{S}^{n-1}}{\arg \max }\mu ^{emp}(w,0)\,,& \label{A1}\\
&\mathrm{rank\,}\tilde{X}_{{I}_0}<\min \big\{ n,\mathrm{rank\,}\tilde{X}\big\}\,,
                 \quad \mathrm{rank\,}{X}_{{I}_0} = \mathrm{rank\,}\tilde{X}_{{I}_0} - 1\,.& \label{A2}
\end{eqnarray}
Then there exist $w_1\in\mathbb{S}^{n-1}$ and $I_1$ with $I_0\subseteq I_1\subseteq\{1,\ldots,N\}$
such that
\begin{eqnarray}
& w_1 \in \mathrm{Ker\,} X_{I_1}^{T}\cap\mathbb{S}^{n-1}, \quad \mu^{emp}(w_1,0)=\mu^{emp}(w_0,0),& \label{S1}\\
&\mathrm{rank\,}X_{I_1} = \mathrm{rank\,}{X}_{I_0} + z,\quad
 \mathrm{rank\,}\tilde{X}_{I_1} = \mathrm{rank\,}\tilde{X}_{I_0} + z & \label{S2}
\end{eqnarray}
for some natural number $z\geq 1$.
\end{lemma}
\begin{proof}
We claim that assumptions (\ref{A1}) and (\ref{A2}) imply that the index set
\[
J_0:=\left\{j\in\{1,\ldots,N\}\,|\,\langle w_0, x^j \rangle > 0 \right\}
\]
is nonempty. Indeed, in case that $J_0=\emptyset $, we would
have that $\mu ^{emp}(w_0,0)=N^{-1}\#I_0$ and that
$\left\langle w_0,x^{i}\right\rangle \leq 0$ for all $i\in \left\{
1,\ldots ,N\right\} $, which amounts to $\mu ^{emp}(-w_0,0)=1$. Then,
since $-w_0\in \mathrm{Ker\,}X_{I_0}^{T}\cap \mathbb{S}^{n-1}$, it would follow that
\begin{equation*}
1=\mu^{emp}(-w_0,0)\leq \mu^{emp}(w_0,0)=N^{-1}\#I_0\leq 1.
\end{equation*}%
Consequently, $N=\#I_0$, hence $\tilde{X}=\tilde{X}_{I_0}$ and we arrive at the contradiction
\begin{equation*}
\mathrm{rank\,}\tilde{X}=\mathrm{rank\,}\tilde{X}_{I_0}<
\min \big\{ n,\mathrm{rank\,}\tilde{X}\big\} \leq \mathrm{rank\,}\tilde{X}.
\end{equation*}%
Therefore, $J_0\neq \emptyset $.

From the assumption $w_0 \in \mathrm{Ker\,} X_{I_0}^T \cap \mathbb{S}^{n-1}$
and from the definitions of $I_0$, $J_0$ we observe that
\begin{equation} \label{esta}
 \mu^{emp} (w_0,0) = N^{-1}\,(\# I_0 + \# J_0). 
\end{equation}
In order to show the existence of some suitable $w_1$
let us consider the following optimization problem:%
\begin{equation}
\min_{w}\left\{ \varphi (w)\,\left\vert \,w\in \mathrm{Ker\,}X_{I_0}^{T}\cap
\mathbb{S}^{n-1},\,\varphi (w)\geq 0\right. \right\} ,\quad
\varphi (w):=\min_{j\in J_0}\langle w,x^{j}\rangle .  \label{optimprob}\end{equation}%
Observe that the feasible set of this problem is nonempty (it contains $%
w_0$) and compact by continuity of $\varphi $. Hence, once more by continuity of $\varphi$, the problem
admits a solution ${w_1}$. Select $j_1\in J_0$ satisfying $%
\langle {w_1},x^{j_1}\rangle =\varphi ({w_1})$ and put ${K}%
:=I_0\cup \left\{ j_1 \right\} $.

Next, we prove that $\varphi (w_1)=0$. Assume to the contrary that $%
\varphi ({w_1})>0$. 
Because $x^{j_1}\notin \mathrm{span}\,\left\{x^{i}\right\}_{i\in I_0}$
by $w_1\in \mathrm{Ker\,}X_{I_0}^{T}$ and $\langle {w_1},x^{j_1}\rangle>0$,
we observe that
\begin{equation}\label{rank1}
   \mathrm{rank\,}X_{K}=\mathrm{rank\,}X_{I_0} + 1.
\end{equation}
Assumption (\ref{A2}) and property (\ref{rank1}) imply that
\begin{equation*}
\dim \mathrm{Ker\,}X_{K}^{T} = n-\mathrm{rank\,} X_{K} = n - \mathrm{rank\,} X_{I_0} -1 > 0,
\end{equation*}%
whence $\mathrm{Ker\,}X_{K}^{T}\cap \mathbb{S}^{n-1}\neq \emptyset $.
Select some $\bar{w}\in \mathrm{Ker\,}X_{K}^{T}\cap \mathbb{S}^{n-1}$
moreover satisfying $\langle {w_1},\bar{w}\rangle \geq 0$ and define%
\begin{equation*}
\bar{w}_{t}:=t\bar{w}+\left( 1-t\right) {w_1}\quad \forall t\in \left[ 0,1\right].
\end{equation*}%
Then, with $\|\cdot\|$ referring to the Euclidean norm, we derive that
\begin{equation}\label{strictpos}
\left\Vert
\bar{w}_{t}\right\Vert >1-t>0\quad\forall t\in (0,1).
\end{equation} 
In particular, recalling that
${w_1}\in \mathrm{Ker\,}X_{I_0}^{T}\cap \mathbb{S}^{n-1}$ is a solution of
(\ref{optimprob}) and that 
\begin{equation*}
\bar{w}\in \mathrm{Ker\,}X_{K}^{T}\cap \mathbb{S}^{n-1}\subseteq 
\mathrm{Ker\,}X_{I_0}^{T}\cap \mathbb{S}^{n-1},
\end{equation*}%
we may define 
\begin{equation*}
\tilde{w}_{t}:=\bar{w}_{t}/\left\Vert \bar{w}_{t}\right\Vert \in \mathrm{Ker\,}%
X_{I_0}^{T}\cap \mathbb{S}^{n-1}\quad \forall t\in (0,1).
\end{equation*}%
Now, since $\lim\limits_{t\downarrow 0}\|\bar{w}_{t}\|=\|w_1\|=1$, we infer that for all
$j\in J_0$, 
\begin{equation*}
\lim\limits_{t\downarrow 0}\left\langle \tilde{w}_{t},x^{j}\right\rangle
=\lim\limits_{t\downarrow 0}\left( t\left\langle \bar{w},x^{j}\right\rangle
+\left( 1-t\right) \left\langle {w_1},x^{j}\right\rangle \right)
/\left\Vert \bar{w}_{t}\right\Vert =\left\langle {w_1},x^{j}\right\rangle \geq
\varphi ({w_1})>0.
\end{equation*}%
Consequently, $\varphi (\tilde{w}_{t})\geq 0$ for small enough $t>0$ which
entails that $\tilde{w}_{t}$ is feasible in problem (\ref{optimprob}) forsmall enough $t>0$. On the other hand, since $\bar{w}\in \mathrm{Ker\,}X_{{K}}^{T}$, we may exploit the relation
$\left\langle \bar{w},x^{j_1}\right\rangle =0$, in order to derive from \eqref{strictpos}  and $\varphi(w_1)>0$ that 
\begin{eqnarray*}
\varphi (\tilde{w}_{t}) &\leq &\left\langle \tilde{w}_{t},x^{j_1}\right\rangle
=\left( t\left\langle \bar{w},x^{j_1}\right\rangle
+\left( 1-t\right) \left\langle {w_1},x^{j_1}\right\rangle \right)
/\left\Vert \bar{w}_{t}\right\Vert \\
&=&\left( 1-t\right)\varphi(w_1)
/\left\Vert \bar{w}_{t}\right\Vert <\varphi ({w_1})
\end{eqnarray*}%
for all $t\in (0,1)$, whence the contradiction that for small enough $t>0$
$\tilde{w}_{t}$ is feasible in problem (\ref{optimprob}) and realizes a strictly smaller objective value than the solution $w_1$. Hence, we have shown that
$\varphi(w_1) = 0$.

From $\langle {w_1},x^{j_1}\rangle = 0$ it follows that
${w_1}\in \mathrm{Ker\,}X_{K}^{T}\cap \mathbb{S}^{n-1}$.
Put 
\[
I_1:=\left\{\left.i\in\{1,\ldots,N\}\,\right|\langle w_1,x^i \rangle = 0\right\}
\]
and obtain that 
\begin{equation}\label {properinclu}
I_0\subset K \subseteq I_1.
\end{equation}
Since $\mathrm{Ker\,}X_K^{T}\subseteq\mathrm{Ker\,}X_{I_0}^{T}$, the relation%
\begin{equation*}
\langle {w_1},x^{j}\rangle \geq \varphi ({w_1})=0 \quad \forall j\in J_0
\end{equation*}%
implies together with equation (\ref{esta}) and assumption (\ref{A1}) that 
\begin{equation*}
\,\mu ^{emp}({w_1},0)\geq N^{-1}\left( \#I_0+\#J_0\right) = \mu
^{emp}(w_0,0)\geq \mu ^{emp}(w_1,0).
\end{equation*}%
Hence, $\mu ^{emp}({w_1},0)=\mu ^{emp}(w_0,0)$. This, along with the definition of $I_1$ shows the two relations claimed in \eqref{S1}. 

In order to verify \eqref{S2}, let finally $I_1 \setminus I_0 = \{k_1,\ldots,k_s\}$ and put $K_0:=I_0$, $K_\ell := I_0 \cup \{k_1,\ldots,k_\ell\}$ for $\ell=1,\ldots, s$.
Obviously, 
\[
 \mathrm{rank\,} \tilde{X}_{K_{\ell}} - \mathrm{rank\,} \tilde{X}_{K_{\ell-1}} \geq \mathrm{rank\,} {X}_{K_{\ell}} - \mathrm{rank\,} {X}_{K_{\ell-1}}
\]
for all $\ell = 1,\ldots,s$, whence
\begin{equation}
 \mathrm{rank\,} \tilde{X}_{I_1} - \mathrm{rank\,} \tilde{X}_{I_0}
 \geq\mathrm{rank\,} {X}_{I_1} - \mathrm{rank\,} {X}_{I_0}=:z \,. \label{z1}
\end{equation}
By  (\ref{rank1}) and (\ref{properinclu}), we have that $ z \geq 1$.
On the other hand, assumption (\ref{A2}) implies  that
\begin{eqnarray}
    \mathrm{rank\,} \tilde{X}_{I_1} - \mathrm{rank\,} \tilde{X}_{I_0} &\leq& \mathrm{rank\,} {X}_{I_1} + 1  - \mathrm{rank\,} \tilde{X}_{I_0} \nonumber\\
    &=& z + \mathrm{rank\,} {X}_{I_0} + 1  - (\mathrm{rank\,} {X}_{I_0} + 1) \,=\, z\,. \label{z2}
\end{eqnarray}
Estimations (\ref{z1}) and (\ref{z2}) show the relations claimed in (\ref{S2}) and we are done.  
\end{proof}

We finish this section by a simple implication which will be needed several times in the proof of the main result below and which uses the notation introduced in Lemma \ref{rankselect}:
\begin{eqnarray}\label{implichain}
\mathrm{Ker}\,X^T_I\cap\mathbb{S}^{n-1}\neq\emptyset&\Longrightarrow&\mathrm{dim}\,\mathrm{Ker}\,X^T_I\geq 1 \nonumber\\
&\Longrightarrow& \mathrm{rank}\,\tilde{X}_I\leq\mathrm{rank}\,X_I+1\leq n\nonumber\\
&\Longrightarrow& \mathrm{rank}\,\tilde{X}_I\leq \min\{n,\mathrm{rank}\,\tilde{X}\}.
\end{eqnarray}

\section{Main Result}

\noindent
After the preparations of the previous section, we are in a
position to derive a formula allowing for the computation of the cap
discrepancy $\Delta$ of any sample on the sphere by enumeration of finitely many easy to calculate local discrepancies $\Delta (w,t)$.
The Theorem divides into a simpler part for the case that the half space realizing the discrepancy does not contain the origin on its boundary (i.e., $t^*\neq 0$ for the couple $(w^*,t^*)$ in Proposition \ref{maxsup}) and a technically more delicate part in case that the origin does belong to that boundary (i.e., $t^*=0$).
\begin{theorem}\label{maintheo}
Let $\left\{ x^{1},\ldots ,x^{N}\right\} \subseteq \mathbb{S}^{n-1}$. For
any $I\subseteq \{1,\ldots ,N\}$ with $I\neq\emptyset$, let $X_{I}$ be the matrix whose columns are 
$x^{i}$ ($i\in I$) and define $\tilde{X}_{I}:=\left( {X_{I}}\atop {-\mathbf{1%
}^{T}}\right) $ as well as $\tilde{X}:=\tilde{X}_{\{1,\ldots ,N\}}$. Consider the following finite
families of index sets:%
\begin{eqnarray*}
\Phi _{1} &:&=\left\{I\subseteq \{1,\ldots ,N\}\,\left|\,1\leq\mathrm{rank\,}\tilde{X}%
_{I}=\#I\leq\min\big\{n,\mathrm{rank\,}\tilde{X}\big\};\,\gamma _{I}<1\right.\right\}, \\
\Phi _{0} &:&=\left\{I\subseteq \{1,\ldots ,N\}\,\left|\,1\leq\mathrm{rank\,}\tilde{X}%
_{I}=\#I=\min\big\{n,\mathrm{rank\,}\tilde{X}\big\};\,%
\gamma _{I}=1 \right.\right\},
\end{eqnarray*}%
where $\gamma _{I}:=\mathbf{1}^{T}\left( \tilde{X}_{I}^{T}\tilde{X}%
_{I}\right) ^{-1}\mathbf{1}$. For $I\in \Phi _{1}\cup \Phi _{0}$ put 
\begin{equation*}
t_{I}:=\left\{ 
\begin{tabular}{cl}
$\left( \frac{1-\gamma _{I}}{\gamma _{I}}\right) ^{1/2}$ & $I\in \Phi _{1}$
\\ 
$0$ & $I\in \Phi _{0}$%
\end{tabular}%
\right. ,\mathbf{\quad }w_{I}:=\left\{ 
\begin{tabular}{cl}
$\frac{1+{t_{I}}^{2}}{t_{I}}X_{I}\left( \tilde{X}_{I}^{T}\tilde{X}%
_{I}\right) ^{-1}\mathbf{1}$ & $I\in \Phi _{1}$ \\ 
$\in \mathrm{Ker\,}X_{I}^{T}\cap \mathbb{S}^{n-1}$ & $I\in \Phi _{0}$%
\end{tabular}%
\right. ,
\end{equation*}%
where the selection of $w_{I}$ in case of $I\in \Phi _{0}$ is arbitrary.
Then, for the cap discrepancy it holds that $\Delta =\max \left\{ \Delta
_{1},\Delta _{0}\right\} $, where 
\begin{eqnarray*}
\Delta _{1}:= &&\left\{ 
\begin{tabular}{cl}
$\max\limits_{I\in \Phi _{1}}\max \left\{ \Delta (w_{I},t_{I}),\Delta
(-w_{I},-t_{I})\right\} $ & if $\Phi _{1}\neq \emptyset $ \\ 
$0$ & else%
\end{tabular}%
\right. ,\\
\Delta _{0}:= &&\left\{ 
\begin{tabular}{cl}
$\max\limits_{I\in \Phi _{0}}\max \left\{ \Delta (w_{I},0),\Delta
(-w_{I},0)\right\} $ & if $\Phi _{0}\neq \emptyset $ \\ 
$0$ & else%
\end{tabular}%
\right..
\end{eqnarray*}
\end{theorem}

\begin{proof}
Let $\left( w^{\ast },t^{\ast }\right) \in \mathbb{S}^{n-1}\times \left[ -1,1%
\right] $ be such that (see Prop.~\ref{maxsup}) 
\[
\Delta =\Delta(w^*,t^*)=\left|\mu ^{emp}\left( w^{\ast },t^{\ast } \right) 
-\mu^{cap}\left( w^{\ast },t^{\ast }\right) \right|.
\]
Since, by Corollary~\ref{corollary}
\begin{equation} \label{empgreatercap}
\mu ^{emp}\left( w^{\ast },t^{\ast }\right) \geq \mu^{cap}\left( w^{\ast },t^{\ast }\right)
\end{equation}
it follows that
\begin{equation}
 \Delta =\mu ^{emp}\left( w^{\ast },t^{\ast }\right) -\mu^{cap}\left( w^{\ast },t^{\ast }\right).  \label{disctotal}
\end{equation}
We define the (disjoint) index sets 
\begin{equation*}
I^{\ast }:=\{i\in \{1,\ldots ,N\}\,|\,\left\langle w^{\ast
},x^{i}\right\rangle =t^{\ast }\}\,,\,\, J^{\ast }
:=\{i\in \{1,\ldots,N\}\,|\,\left\langle w^{\ast },x^{i}\right\rangle >t^{\ast }\}.
\end{equation*}
From Proposition \ref{nonvoid}, we infer that $I^*\neq\emptyset$. Let 
\begin{equation*}
S:=\left\{ \left( w,t\right) \in \mathbb{S}^{n-1}\times \left[ -1,1\right]
\left\vert 
\begin{array}{cc}
\left\langle w,x^{i}\right\rangle =t & i\in I^{\ast } \\ 
\left\langle w,x^{i}\right\rangle >t & i\in J^{\ast } \\ 
\left\langle w,x^{i}\right\rangle <t & i\in \left\{ 1,\ldots ,N\right\}
\backslash (I^{\ast }\cup J^{\ast })%
\end{array}%
\right. \right\} .
\end{equation*}%
The definitions of $\Delta$ and $(w^{\ast },t^{\ast })$ imply along with (\ref{empgreatercap}) that 
\begin{equation}
(w^{\ast },t^{\ast })\in \underset{(w,t)\in \mathbb{S}^{n-1}\times
\left[-1,1\right] }{\arg \max } \mu ^{emp}(w,t)-\mu ^{cap}(w,t).  \label{am}
\end{equation}%
Since $(w^{\ast },t^{\ast })\in S$ it follows that even 
\begin{equation*}
(w^{\ast },t^{\ast })\in \,\underset{(w,t)\in S}{\arg \max }\,\,
\mu^{emp}(w,t)-\mu ^{cap}(w,t).
\end{equation*}%
We observe that $\mu ^{emp}(w,t)=N^{-1}(\#I^{\ast }+\#J^{\ast })={\rm const}$ for all 
$(w,t)\in S$. Hence,
\begin{equation*}
(w^{\ast },t^{\ast })\in \,\underset{(w,t)\in S}{\arg \min }\,\mu
^{cap}(w,t) \, .
\end{equation*}%
Because, $\mu ^{cap}(w,t)$ depends on $t$ only and is
monotonically decreasing with $t$ (see(\ref{capmeasure})), $(w^{\ast
},t^{\ast })$ is a solution of the optimization problem
\begin{equation*}
\max_{w,t}\left\{ t\,\left\vert \,%
\begin{array}{lc}
\left\langle w,x^{i}\right\rangle =t\,\, & i\in I^{\ast } \\ 
\left\langle w,x^{i}\right\rangle >t\,\, & i\in J^{\ast } \\ 
\left\langle w,x^{i}\right\rangle <t\,\, & i\in \{1,\ldots ,N\}\setminus
(I^{\ast }\cup J^{\ast }) \\ 
\left\langle w,w\right\rangle =1 & 
\end{array}%
\right. \right\} .
\end{equation*}%
Note, that the constraint $t\in [-1,1]$ is implicitly contained in the equality constraints above. Next, choose a subset $\bar{I}^{\ast }\subseteq I^{\ast }$ such that 
\begin{equation}
\#\bar{I}^{\ast }=\mathrm{rank\,} \tilde{X}_{\bar{I}^{\ast }} = \mathrm{rank\,} \tilde{X}_{I^\ast}
\label{fullrank}
\end{equation}%
(using the notation introduced in the statement of the Theorem). 
By Lemma~\ref{setopt}, the optimization problem above can be
reformulated as 
\begin{equation}
\max_{w,t}\left\{ t\,\left\vert \,%
\begin{array}{lc}
\left\langle w,x^{i}\right\rangle =t\,\, & i\in \bar{I}^{\ast } \\ 
\left\langle w,x^{i}\right\rangle >t\,\, & i\in J^{\ast } \\ 
\left\langle w,x^{i}\right\rangle <t\,\, & i\in \{1,\ldots ,N\}\setminus
(I^{\ast }\cup J^{\ast }) \\ 
\left\langle w,w\right\rangle =1 & 
\end{array}%
\right. \right\} .  \label{opt_red}
\end{equation}%
Since $\left\langle w^{\ast },x^{i}\right\rangle >t^{\ast }$ for $i\in
J^{\ast }$ and $\left\langle w^{\ast },x^{i}\right\rangle <t^{\ast }$ for $%
i\in \{1,\ldots ,N\}\setminus (I^{\ast }\cup J^{\ast })$ and $(w^{\ast
},t^{\ast })$ is a solution of the optimization problem
(\ref{opt_red}), it follows that $(w^{\ast },t^{\ast })$ must be a local solution
of the optimization problem
\begin{equation}
\max_{w,t}\left\{ t\,\left\vert \,\left\langle
w,x^{i}\right\rangle =t\,\quad \left( i\in \bar{I}^{\ast }\right)
;\left\langle w,w\right\rangle =1\right. \right\} .  \label{optred2}
\end{equation}%
By (\ref{fullrank}), this problem satisfies the assumption of
Lemma~\ref{crit} with $X_*:=X_{\bar{I}^{\ast }}$ and
$k:=\#\bar{I}^{\ast }$. According to that Lemma we have that 
$0<\gamma _{\bar{I}^{\ast }}\leq 1$ with $\gamma_{I}$ as introduced 
in the statement of this Theorem.

In the case of $\gamma _{\bar{I}^{\ast }}<1$ it follows from Lemma~\ref{crit} (last statement), that $t^*\neq 0$.
Then, by feasibility of $(w^*,t^*)$ in \eqref{optred2}, we have that
\[
(t^*)^{-1}\tilde{X}_{\bar{I}^*}^Tw^*={\bf 1}\quad (=(1,\ldots ,1)\in\mathbb{R}^{\#\bar{I}^*}).
\]
Consequently, $-\mathbf{1} \in \mathrm{range\,} X_{\bar{I}^{\ast }}^T$, and thus,
\[
\mathrm{rank\,}\tilde{X}_{\bar{I}^{\ast }} =\mathrm{rank\,} {X_{\bar{I}^*}\choose\mathbf{-1}^T}
=\mathrm{rank\,} (X^T_{\bar{I}^*}\mid\mathbf{-1})=\mathrm{rank\,} X^T_{\bar{I}^*}\leq n.
\]
Since also \,$\mathrm{rank\,}\tilde{X}_{\bar{I}^{\ast }}\leq\mathrm{rank}\,\tilde{X}$, we have shown that
$\bar{I}^{\ast }\in \Phi _{1}$.
Therefore, with the
definitions of $t_{I},w_{I}$ in the statement of this Theorem, we infer from
Lemma~\ref{crit} that
$\left( w^{\ast },t^{\ast }\right) \in \left\{ \left( w_{\bar{I}^{\ast }},t_{%
\bar{I}^{\ast }}\right) ,\left( -w_{\bar{I}^{\ast }},-t_{\bar{I}^{\ast
}}\right) \right\}$.
Thus, 
\begin{equation}
\Delta =\Delta (w^{\ast },t^{\ast }) \leq \max \left\{ \Delta (w_{\bar{I}%
^{\ast }},t_{\bar{I}^{\ast }}),\Delta (-w_{\bar{I}^{\ast }},-t_{\bar{I}%
^{\ast }})\right\}
\leq\Delta_1 \label{delta1}  
\end{equation}
with $\Delta_1$ as introduced in the statement of this Theorem.

The remaining part of this proof is devoted to the case $\gamma _{\bar{I}^{\ast }}=1$. From Lemma~\ref%
{crit} we observe that $t^{\ast }=0$, and, $\mathrm{rank\,} X_{\bar{I}^\ast} = \#\bar{I}^{\ast } - 1$.
The second equality in (\ref{fullrank}) along with $\bar{I}^\ast \subseteq I^\ast$ yields that
$\mathrm{rank\,} X_{{I}^\ast} = \mathrm{rank\,} X_{\bar{I}^\ast}$. Hence, the first equality in \eqref{fullrank} provides the relation
\begin{equation}\label{kernX}
 \mathrm{rank\,} X_{{I}^\ast} = \mathrm{rank\,} \tilde{X}_{{I}^\ast} - 1.
\end{equation}
Moreover, by definition of $I^{\ast }$, one has that
$w^{\ast}\in \mathrm{Ker\,}X_{I^{\ast }}^{T}\cap \mathbb{S}^{n-1}$, so that 
\begin{equation}
w^{\ast }\in \!\!\underset{w\in \mathrm{Ker\,}X_{I^{\ast }}^{T}\cap \mathbb{S}%
^{n-1}}{\arg \max } \!\! \mu ^{emp}(w,0)-\mu ^{cap}(w,0) \,=%
\!\! \underset{w\in \mathrm{Ker\,}X_{I^{\ast }}^{T}\cap \mathbb{S}^{n-1}}{\arg
\max } \mu ^{emp}(w,0)-\tfrac{1}{2}  \label{delmuemp}
\end{equation}%
as a consequence of (\ref{am}). Therefore, 
\begin{equation}
w^{\ast }\in A:=\underset{w\in \mathrm{Ker\,}%
X_{I^{\ast }}^{T}\cap \mathbb{S}^{n-1}}{\arg \max }\mu ^{emp}(w,0).
\label{claimmuemp}
\end{equation}%
Since $\mu ^{emp}(w^{\ast},0)\geq \tfrac{1}{2}$, it holds that
\begin{equation}
\mu ^{emp}(w,0)\geq \tfrac{1}{2}\quad \forall w\in A .  \label{largeronehalf}
\end{equation}%

By \eqref{implichain}, $w^{\ast }\in \mathrm{Ker\,}X_{I^{\ast }}^{T}\cap \mathbb{%
S}^{n-1}$ implies that $\mathrm{rank\,}\tilde{X}_{I^{\ast }} \leq \min \big\{ n,\mathrm{rank\,}\tilde{X}\big\}$.
We claim the existence of some index set $\hat{I}$ and of some vector $\hat{w%
}$ such that 
\begin{equation}\label{claim} 
\begin{aligned}
&I^{\ast } \subseteq \hat{I}\subseteq \left\{ 1,\ldots ,N\right\} ,\,\,%
\mathrm{rank\,}\tilde{X}_{\hat{I}}=\min \big\{ n,\mathrm{rank\,}\tilde{X}\big\} , \\
&\hat{w} \in \mathrm{Ker\,}X_{\hat{I}}^{T}\cap \mathbb{S}^{n-1},\,\,\mu
^{emp}(\hat{w},0)=\mu ^{emp}(w^{\ast },0).
\end{aligned}
\end{equation}
If $\mathrm{rank\,}\tilde{X}_{I^{\ast }} = \min \big\{ n,\mathrm{rank\,}\tilde{X}\big\}$, then we may choose 
$\hat{I}:=I^*$ and $\hat{w}:=w^*$ in \eqref{claim}.
Otherwise,  $\mathrm{rank\,}\tilde{X}_{I^{\ast }} < \min \big\{ n,\mathrm{rank\,}\tilde{X}\big\}$ and we make use of Lemma~\ref{rankselect} starting with the data $I_0:=I^\ast$ and $w_0 := w^\ast$. Observe that by virtue of \eqref{kernX} and \eqref{claimmuemp}, $I_0$ and $w_0$ satisfy the assumptions \eqref{A1} and \eqref{A2} of that Lemma. Accordingly, we derive the existence of some index set $I_1\supseteq I_0$ and $w_1$ satisfying the relations \eqref{S1} and \eqref{S2}.
In particular, $\mathrm{Ker\,}X_{I_1}^{T}\subseteq\mathrm{Ker\,}X_{I_0}^{T}$, whence both relations in
\eqref{S1} yield that
\[
w_1\in \underset{w\in \mathrm{Ker\,}X_{I_1}^{T}\cap \mathbb{S}
^{n-1}}{\arg \max }\mu ^{emp}(w,0).
\]
Moreover, we infer from \eqref{A2} and \eqref{S2} that $\mathrm{rank\,}X_{I_1} =\mathrm{rank\,}\tilde{X}_{I_1}-1$ and from the first relation in \eqref{S1} and \eqref{implichain} that $\mathrm{rank\,}\tilde{X}_{I_1}\leq\min \big\{ n,\mathrm{rank\,}\tilde{X}\big\}$. 

Now,
if $\mathrm{rank\,}\tilde{X}_{I_1} = \min \big\{ n,\mathrm{rank\,}\tilde{X}\big\}$, then we may choose 
$\hat{I}:=I_1$ and $\hat{w}:=w_1$ in \eqref{claim} due to \eqref{S1} and $w_0=w^*$. Otherwise, $\mathrm{rank\,}\tilde{X}_{I_1} < \min \big\{ n,\mathrm{rank\,}\tilde{X}\big\}$
and so the assumptions \eqref{A1} and \eqref{A2} of Lemma \ref{rankselect} are also satisfied for $I_1$ and $w_1$ instead of $I_0$ and $w_0$. This allows us to apply  Lemma \ref{rankselect} again. In this way, a sequence 
of index sets $I_k$ and of points $w_k$ ($k=1,2,\ldots$) is obtained for which $I^*=I_0\subseteq I_k$ and by \eqref{S1} and \eqref{S2}
\[
w_k \in \mathrm{Ker\,} X_{I_k}^{T}\cap\mathbb{S}^{n-1}, \,\, \mu^{emp}(w_k,0)=\mu^{emp}(w_0,0),\,\,
 \mathrm{rank\,}\tilde{X}_{I_k} = \mathrm{rank\,}\tilde{X}_{I_{k-1}} + z_k,
\]
where $z_k\in\mathbb{N}$ and $z_k\geq 1$. Since $\mathrm{rank\,}\tilde{X}_{I_k} \leq \min \big\{ n,\mathrm{rank\,}\tilde{X}\big\}$ by \eqref{implichain},
the last relation implies that, after finitely many steps, we arrive at the situation $\mathrm{rank\,}\tilde{X}_{I_k} = \min \big\{ n,\mathrm{rank\,}\tilde{X}\big\}$, so that we may define $\hat{I}:=I_k$ and $\hat{w}:=w_k$ in \eqref{claim}. This finishes the proof of \eqref{claim}.


Next, from \eqref{disctotal} and \eqref{claim} we know that
\begin{equation}
\Delta =\mu ^{emp}(w^{\ast },0)-\tfrac{1}{2}=\mu ^{emp}(\hat{w},0)-\tfrac{1}{2} .  \label{delest}
\end{equation}
This relation shows that $(\hat{w},0)$ realizes the discrepancy $\Delta$ as much as  $(w^*,0)$.
Therefore, we may assume that  $(w^*,t^*)$ is $(\hat{w},0)$ in the beginning of our proof until \eqref{optred2}. In particular, analogously to the index set $I^*$ introduced there, we define
\[
I_{\ast }:=\{i\in \{1,\ldots ,N\}\mid\langle\hat{w},x^{i}\rangle =0\}\]
Following the previous arguments from \eqref{fullrank} to \eqref{optred2}, we may find an index set $\bar{I}_*\subseteq I_*$ such that
\begin{equation}
\#\bar{I}_*=\mathrm{rank\,} \tilde{X}_{\bar{I}_*} = \mathrm{rank\,} \tilde{X}_{I_*}.\label{optred3}
\end{equation}
In particular,
\begin{equation}\label{whatkern}
 \hat{w} \in \mathrm{Ker}\,X^T_{\bar{I}_*}\cap\mathbb{S}^{n-1}.
\end{equation}
Moreover, $(\hat{w},0)$ is a local solution of the optimization problem
\begin{equation}\label{optred4}
\max_{w,t}\left\{ t\,\left\vert \,\left\langle
w,x^{i}\right\rangle =t\,\quad \left( i\in \bar{I}_*\right)
;\left\langle w,w\right\rangle =1\right. \right\} . 
\end{equation}
By \eqref{optred3}, this problem satisfies the assumption of Lemma \ref{crit} with $X_*:=X_{\bar{I}_*}$ and $k:=\#\bar{I}_*$. According to that Lemma (last statement) we have that $\gamma _{\bar{I}_*}=1$ with $\gamma_{I}$ as introduced 
in the statement of this Theorem.
Applying \eqref{implichain} to $\hat{w}\in \mathrm{Ker}\,X^T_{I_*}\cap\mathbb{S}^{n-1}$, we observe that 
\[
\mathrm{rank\,} \tilde{X}_{I_*}\leq \min \big\{ n,\mathrm{rank\,}\tilde{X}\big\}.
\]
On the other hand, since $\hat{I}\subseteq I_*$ by \eqref{claim} and by definition of $I_*$, the rank relation in \eqref{claim} leads to
\[
\mathrm{rank\,} \tilde{X}_{I_*}\geq\mathrm{rank\,} \tilde{X}_{\hat{I}}=\min \big\{ n,\mathrm{rank\,}\tilde{X}\big\},
\] 
whence, along with \eqref{optred3}
\begin{equation}\label{lastrel}
\#\bar{I}_*=\mathrm{rank\,} \tilde{X}_{\bar{I}_*} =\min \big\{ n,\mathrm{rank\,}\tilde{X}\big\}.
\end{equation}
Summarizing, we have shown that $\bar{I}_*\in\Phi_0$.

If in \eqref{lastrel} $\mathrm{rank\,}\tilde{X}_{\bar{I}_*}=\mathrm{rank\,}\tilde{X}$, then 
 there exist coefficients
$\lambda_j^i$ such that
\[
{x^i\choose -1}=\sum_{j\in\bar{I}_*}\lambda_j^i \,{x^j\choose -1}\quad\forall i=1,\ldots ,N.
\]
Therefore, we have for all $ i=1,\ldots ,N$ and all $w\in\mathrm{Ker}\,X^T_{\bar{I}_*}$ that
\[
\langle x^i,w\rangle=\left\langle{x^i\choose -1},{w\choose 0}\right\rangle=
\sum_{j\in\bar{I}_*}\lambda_j^i \,\left\langle {x^j\choose -1},{w\choose 0}\right\rangle=
\sum_{j\in\bar{I}_*}\lambda_j^i \,\langle x^j,w\rangle=0.
\]
This amounts to saying that $\mu^{emp}(w,0)=1$ for all these $w$ and, so, 
\[
\Delta(w,0)=|\mu^{emp}(w,0)-\mu^{cap}(w,0)|=1/2\quad\forall 
w\in\mathrm{Ker}\,X^T_{\bar{I}_*}\cap\mathbb{S}^{n-1}.
\]
We conclude from \eqref{whatkern} and \eqref{delest} that
\[
\Delta =\left\vert \mu ^{emp}(\hat{w},0)-\tfrac{1}{2}\right\vert=\Delta (\hat{w},0)=1/2=\Delta (w,0)\quad\forall w\in\mathrm{Ker}\,X^T_{\bar{I}_*}\cap\mathbb{S}^{n-1}.
\]
Therefore, the value of $\Delta (w_{\bar{I}_*},0)$ in the definition of $\Delta_0$ (see statement of this Theorem) does
not depend on the choice of $w_{\bar{I}_*}\in\mathrm{Ker}\,X^T_{\bar{I}_*}\cap\mathbb{S}^{n-1}$. It follows from 
$\bar{I}_*\in\Phi_0$ that $\Delta\leq\Delta_0$.

Otherwise, if in \eqref{lastrel} $\mathrm{rank\,}\tilde{X}_{\bar{I}_*} 
=n$, then $\dim \mathrm{Ker\,}\tilde{X}_{\bar{I}_*}^{T}=1$ and so there exists some $(\tilde{w},\tilde{t})\in\mathbb{S}^n$ with $\mathrm{Ker\,}\tilde{X}_{\bar{I}_*}^{T}={\rm span}\{(\tilde{w},\tilde{t})\}$. Let 
$w\in\mathrm{Ker}\,X^T_{\bar{I}_*}\cap\mathbb{S}^{n-1}$ be arbitrary.
Then, $(w,0)\in\mathrm{Ker\,}\tilde{X}_{\bar{I}_*}^{T}$ and, hence, there is some $\lambda\in\mathbb{R}$ with
$(w,0)=\lambda (\tilde{w},\tilde{t})$. Clearly, $\lambda\neq 0$ by $w\in\mathbb{S}^{n-1}$. It follows that $\tilde{t}=0$, whence $\tilde{w}\in\mathbb{S}^{n-1}$ and $|\lambda|=1$. Therefore, $w=\pm\tilde{w}$. Thus, we have shown that 
$w\in\{\tilde{w},-\tilde{w}\}$ for all $w\in\mathrm{Ker}\,X^T_{\bar{I}_*}\cap\mathbb{S}^{n-1}$. On the other hand, 
$\hat{w}\in\{\tilde{w},-\tilde{w}\}$ by \eqref{whatkern}. Therefore,
\begin{eqnarray*}
\Delta &\!=\!& \left\vert \mu ^{emp}(\hat{w},0)-\tfrac{1}{2}\right\vert \,=\, \Delta (%
\hat{w},0) \leq \max \left\{ \Delta (\tilde{w},0),\Delta (-\tilde{w},0)\right\} \\
        &\!=\!& \max \left\{ \Delta (w,0),\Delta (-w,0)\right\}\qquad\forall w\in\mathrm{Ker\,}X^T_{\bar{I}_*}\cap \mathbb{S}^{n-1}.
\end{eqnarray*}
As in the previous case, the value of $\Delta (w_{\bar{I}_*},0)$ in the definition of $\Delta_0$ does
not depend on the choice of $w_{\bar{I}_*}\in\mathrm{Ker}\,X^T_{\bar{I}_*}\cap\mathbb{S}^{n-1}$. Again, 
$\bar{I}_*\in\Phi_0$ implies that $\Delta\leq\Delta_0$.

Summarizing, our proof has shown by case distinction that necessarily $\Delta\leq\Delta_1$ (see \eqref{delta1}) or $\Delta\leq\Delta_0$. Therefore, $\Delta \leq \max \{\Delta
_{1},\Delta _{0}\}$.  On the other hand, each of the quantities 
$\Delta _{1},\Delta _{0}$ is either zero or corresponds to a
concrete value $\Delta (w,t)$ for some $w\in \mathbb{S}^{n-1}$ and $\,t\in %
\left[ -1,1\right] $. Hence, in any case $\max \{\Delta _{1},\Delta
_{0}\}\leq \Delta $. This finishes the proof.  
\end{proof}
 
We want to conclude this section with some algorithmic remarks. The formula provided by the main theorem is appropriate for easy implementation.
To compute the cap discrepancy for a given point set one has to consider all possible selections $I\subseteq\{1,\ldots,N\}$ with cardinality less than or equal to $\min\{n,\mathrm{rank\,}\tilde{X}\}$
and one has to check whether the selection is included in one of the two sets $\Phi_1$ or $\Phi_0$. This check implies first a verification of $\mathrm{rank}~\tilde X_I$, and secondly,
if applicable, the computation of $\gamma_I =\mathbf{1}^{T}\left( \tilde{X}_{I}^{T}\tilde{X}_{I}\right) ^{-1}\mathbf{1}$. For these selected $I$ one has to compute the local discrepancy
by the given formulas. Finally, the discrepancy is found as the maximum of the considered local discrepancies.
A Matlab implementation of the enumeration formula for the spherical cap discrepancy provided by the Theorem is accessible through the link:
\href{https://www.wias-berlin.de/people/heitsch/capdiscrepancy}{https://www.wias-berlin.de/people/heitsch/capdiscrepancy}.

We observe that the cardinality of index sets to be checked in the
proven formula is at most%
\begin{equation*}
\sum\limits_{i=1}^{\min\{n,\mathrm{rank\,}\tilde{X}\}}\binom{N}{i}.
\end{equation*}
Whether calculating the spherical cap discrepancy is NP-hard (or W[1]-hard) is left
open for future work.
Clearly, this aspect of complexity limits the application of the formula to
low-dimensional spheres and moderate sample sizes. Hence, it will not be
suitable for verifying asymptotic aspects of sampling schemes. On the other
hand, it may be used to correctly calibrate the efficiency of sampling
schemes within a certain range of the sample size. 

\section{Numerical Illustration}

\noindent
In this section we illustrate the application of the derived formula for the spherical cap discrepancy to spheres $\mathbb{S}^2$ to $\mathbb{S}^5$ with sample sizes reaching
from 2000 to 100 depending on dimension. 
Samples were generated by normalizations of Monte Carlo simulated independent Gaussian distributions which are approximations of the uniform distribution on the sphere.
For the sake of comparison, we oppose the results to the application of an easily computable lower estimate of the discrepancy as it was used, e.g., in \cite{brauchart}: Given a sample $\{x^1,\ldots ,x^N\}$, we clearly have that
\[
\tilde{\Delta}:=\max_{i=1,\ldots ,N}\sup_{t\in \left[ -1,1\right] }\left\vert \mu
^{emp}\left( x^{i},t\right) -\mu ^{cap}\left( x^{i},t\right) \right\vert
\leq \Delta .
\]
\begin{figure}[h]
\includegraphics[width=0.499\textwidth]{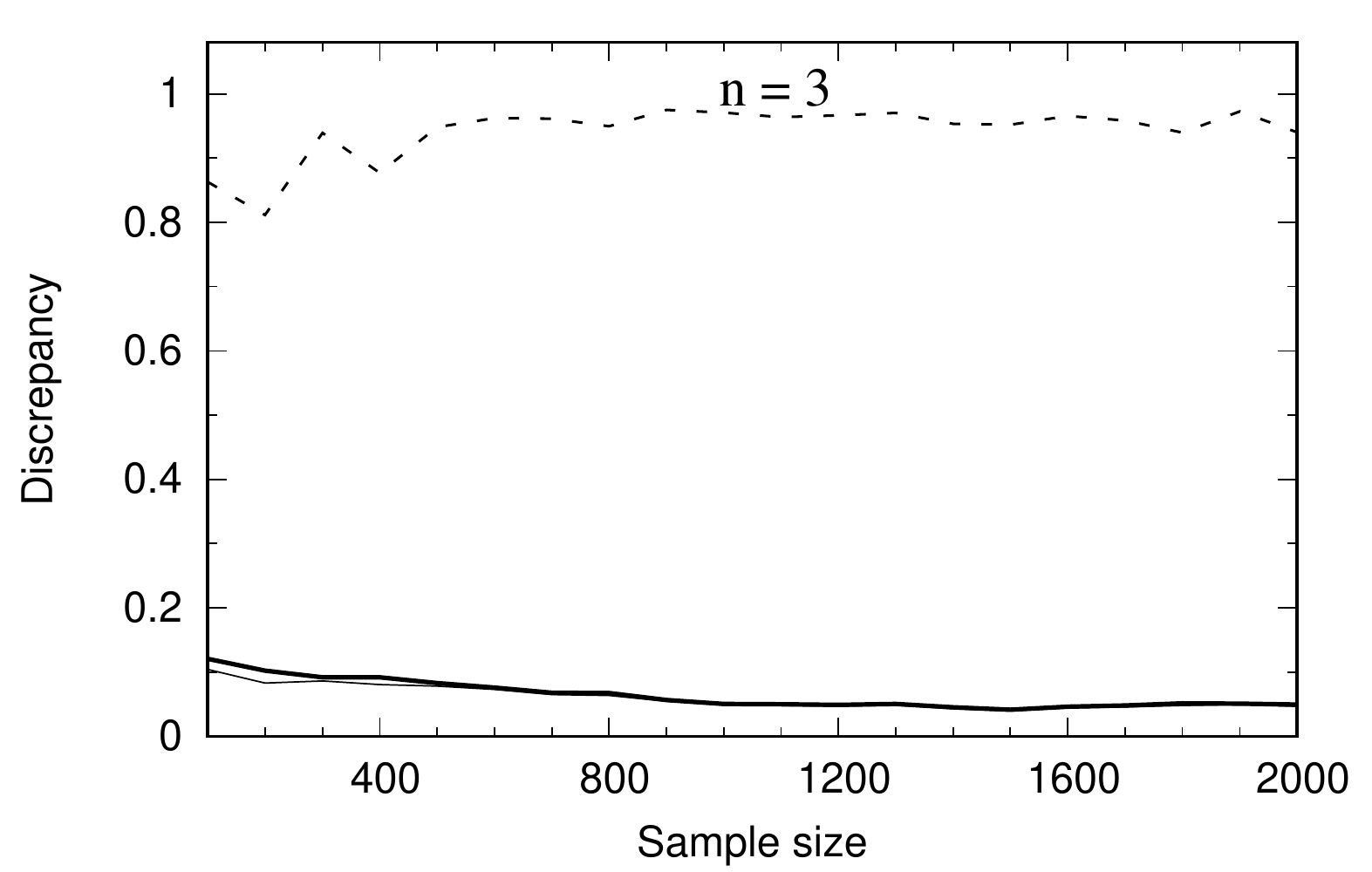}
\includegraphics[width=0.499\textwidth]{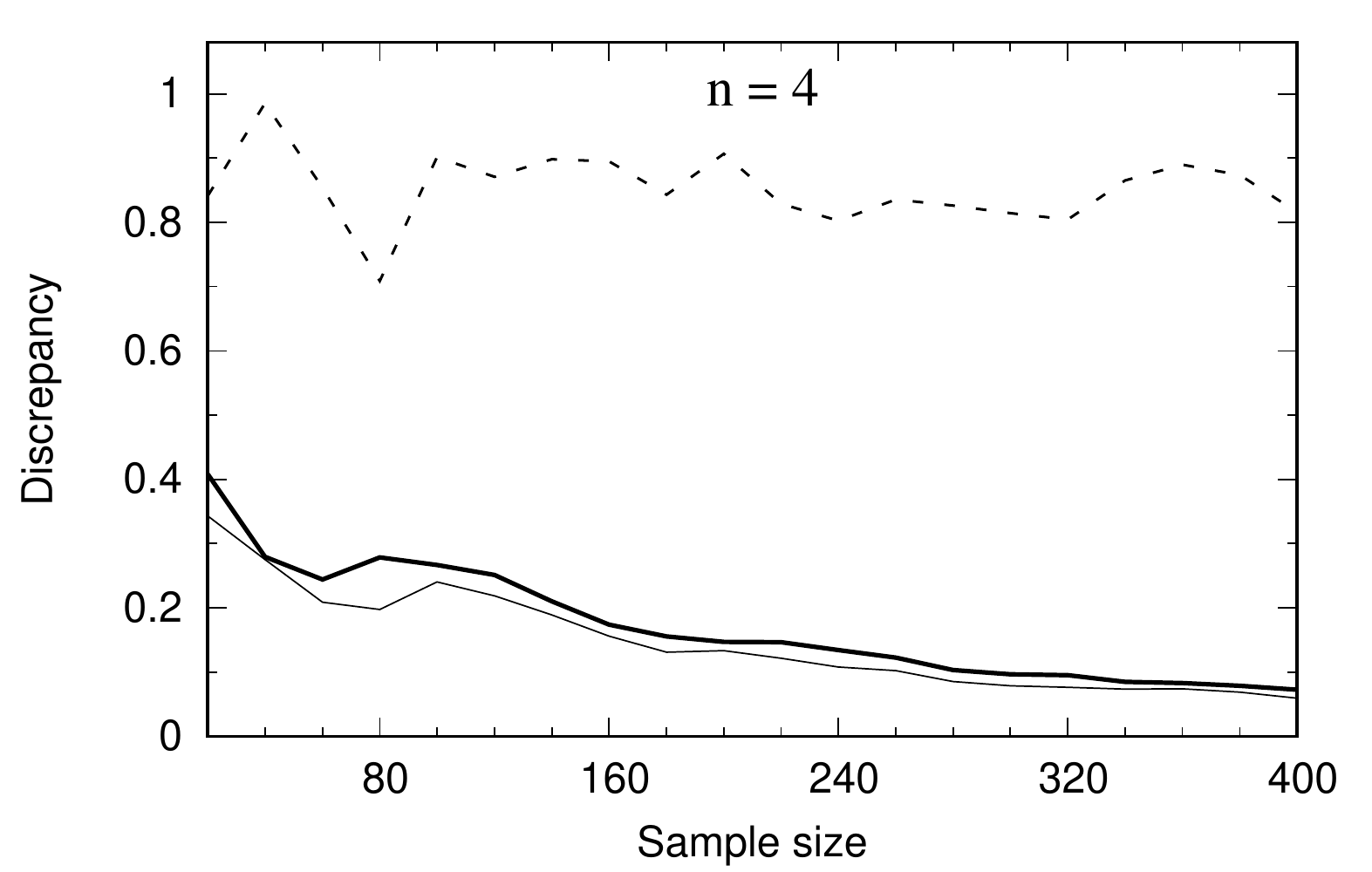}\\
\includegraphics[width=0.499\textwidth]{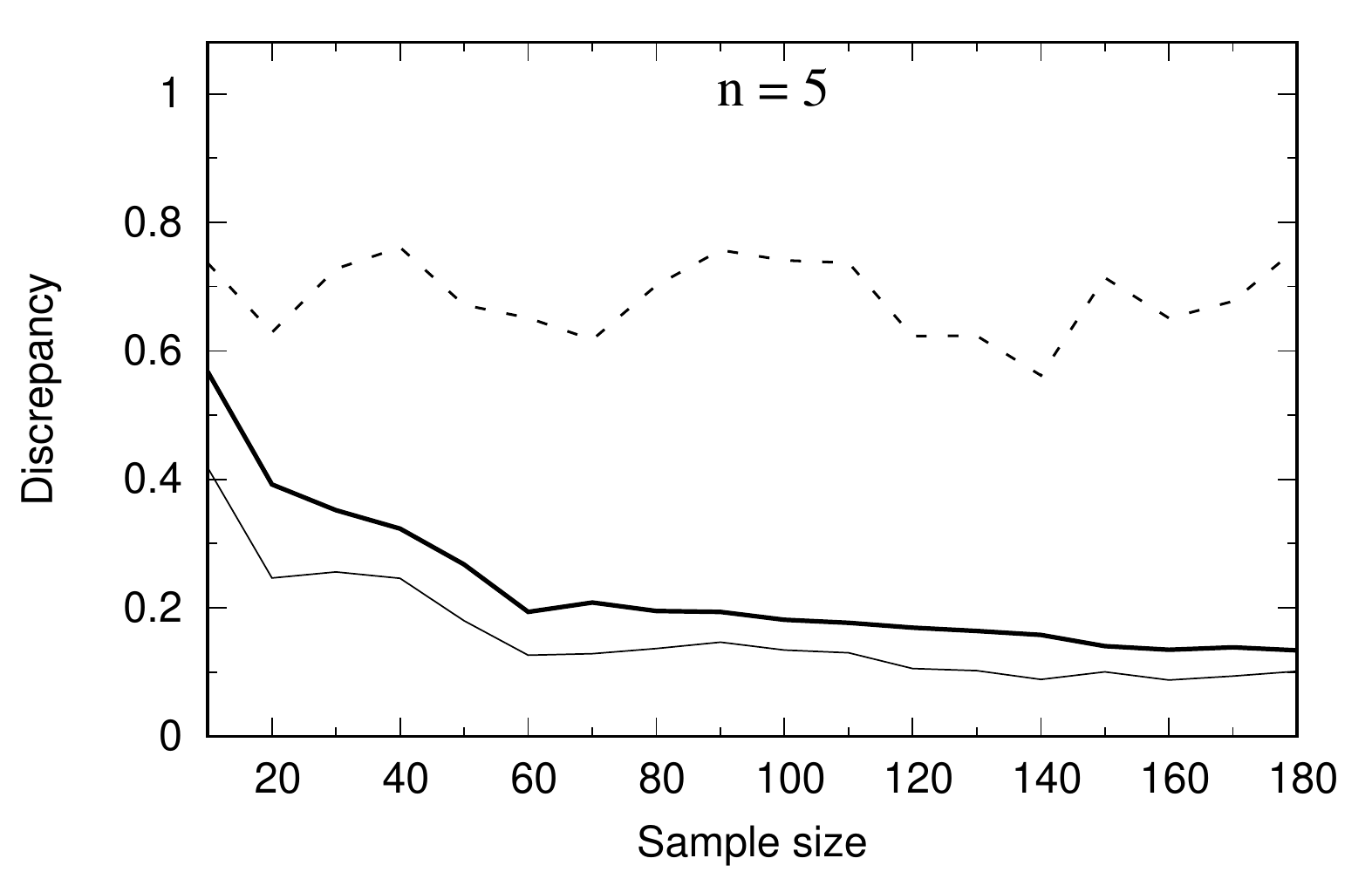}
\includegraphics[width=0.499\textwidth]{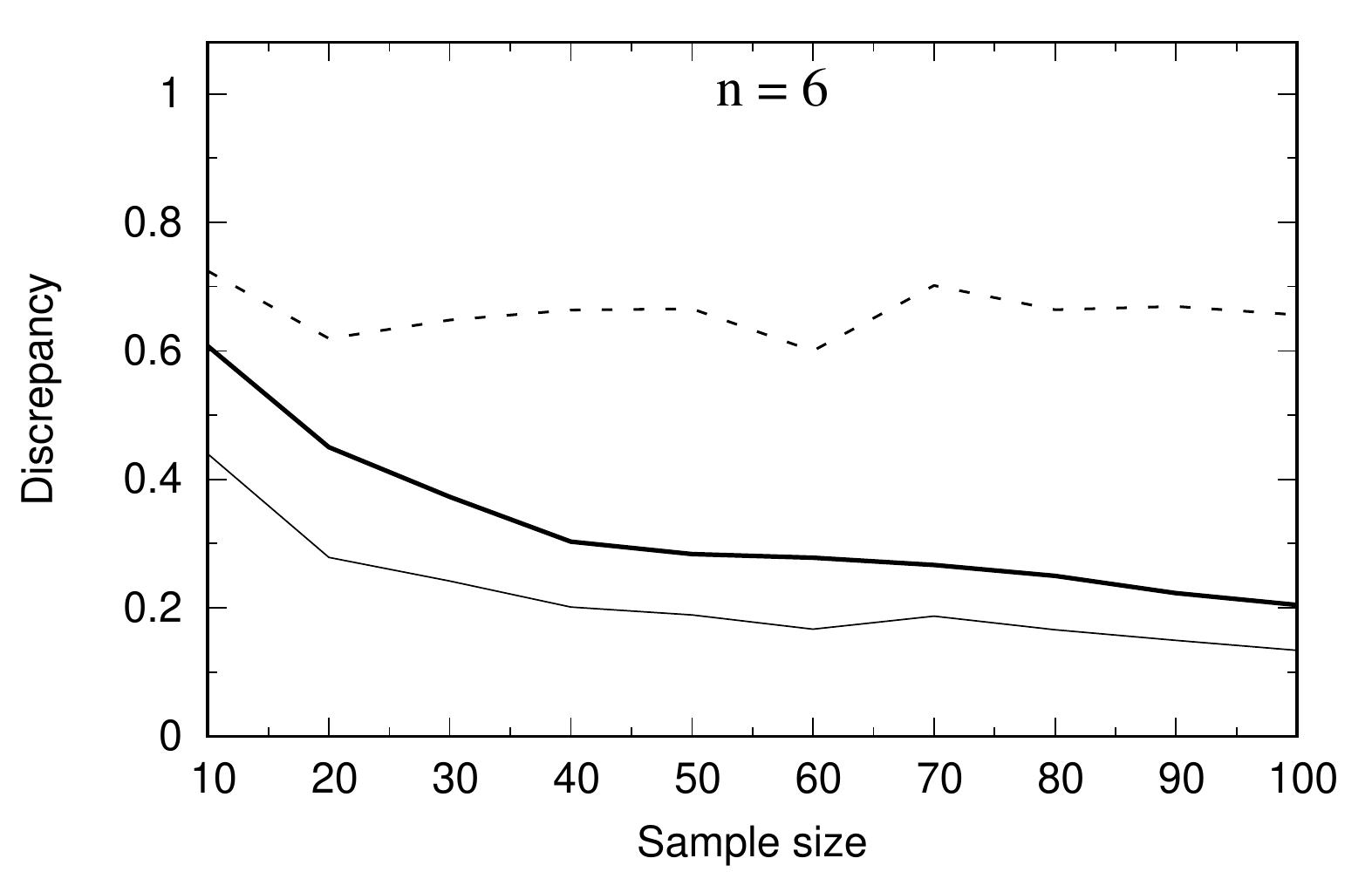}
\caption{Plot of discrepancy $\Delta$ (thick line), of its lower estimate $\tilde{\Delta}$ (thin line) and of the ratio $\tilde{\Delta}/\Delta$ (dashed line) for different dimensions ($n=3,4,5,6$) and sample sizes.}
\label{diagram}
\end{figure}

\noindent
Fig. \ref{diagram} shows the numerical results. We observe the following trends:
\begin{itemize}
\item Both, $\Delta$ and $\tilde{\Delta}$ are decreasing with increasing sample size.
\item The absolute difference between $\Delta$ and $\tilde{\Delta}$ decreases with the sample size.
\item The absolute difference between $\Delta$ and $\tilde{\Delta}$ increases with the dimension of the sphere.
\item The ratio between $\tilde{\Delta}$ and $\Delta$ is basically constant for variable sample size in each dimension of the sphere (with different values of the constant).
\item The constant itself is decreasing with the dimension of the sphere.\end{itemize}
In particular, it seems that the discrepancy can be replaced by its lower estimate without loss of information in $\mathbb{S}^2$ starting from a sample size of approximately 500. For larger dimension or sample size, it appears that at least the decay rate with respect to the sample size is well reflected by the lower estimate (approximately constant ratio with the true discrepancy), while the deviation from the true discrepancy becomes significant.

\begin{table}[h]
 \setlength{\tabcolsep}{2.0\tabcolsep}
 \begin{tabular*}{\textwidth}{cccccccccc}
  \hline
  \rule{0cm}{2.5ex}{}
      $N$        && 10  & 20 &  40 &    100 &     180  &    400  &   1000  &    2000 \\
  \hline
  \rule{0cm}{2.5ex}{}
  $\mathbb{S}^2$ &&  0  &  0 &   0 &      3 &       30 &    499  &  15924  &  252313 \\
  $\mathbb{S}^3$ &&  0  &  0 &   2 &    139 &     1911 &  77449  &       -    &    -   \\
  $\mathbb{S}^4$ &&  0  &  0 &  14 &   1932 &    48134 &    -    &       -    &    -   \\
  $\mathbb{S}^5$ &&  0  &  1 &  92 &  32658 &     -    &    -    &       -    &    -   \\
  \hline
 \end{tabular*}
 \caption{CPU time (in seconds) of the enumeration formula for selected instances.}
 \label{table1}
\end{table}
%

All computations are performed on a standard computer with single CPU (3.2 GHz).
Table~\ref{table1} displays the CPU time for selected instances of the tested range of $N$ and $n$.
%

In order to illustrate even more directly the application of the proven formula, we provide a comparison of 4 sampling schemes on $\mathbb{S}^2$ for small sample sizes ($\leq 1000$).
The first two methods are based on the already mentioned fact that the normalization to unit length of a standard Gaussian distribution 
$\mathcal{N}(0_m,I_m)$ yields a uniform distribution on $\mathbb{S}^{m-1}$. Therefore, we may simulate the Gaussian distribution
via Monte Carlo (MC) or via Quasi-Monte Carlo (QMC). As an alternative, we follow the proposal in 
\cite{brauchart}, to use the equal-area Lambert transform from the unit square to $\mathbb{S}^2$, again for MC and QMC. For QMC, we applied in both cases Sobol' sequences as a special case of low-discrepancy sequences.  

\begin{figure}[h]
\includegraphics[width=0.499\textwidth]{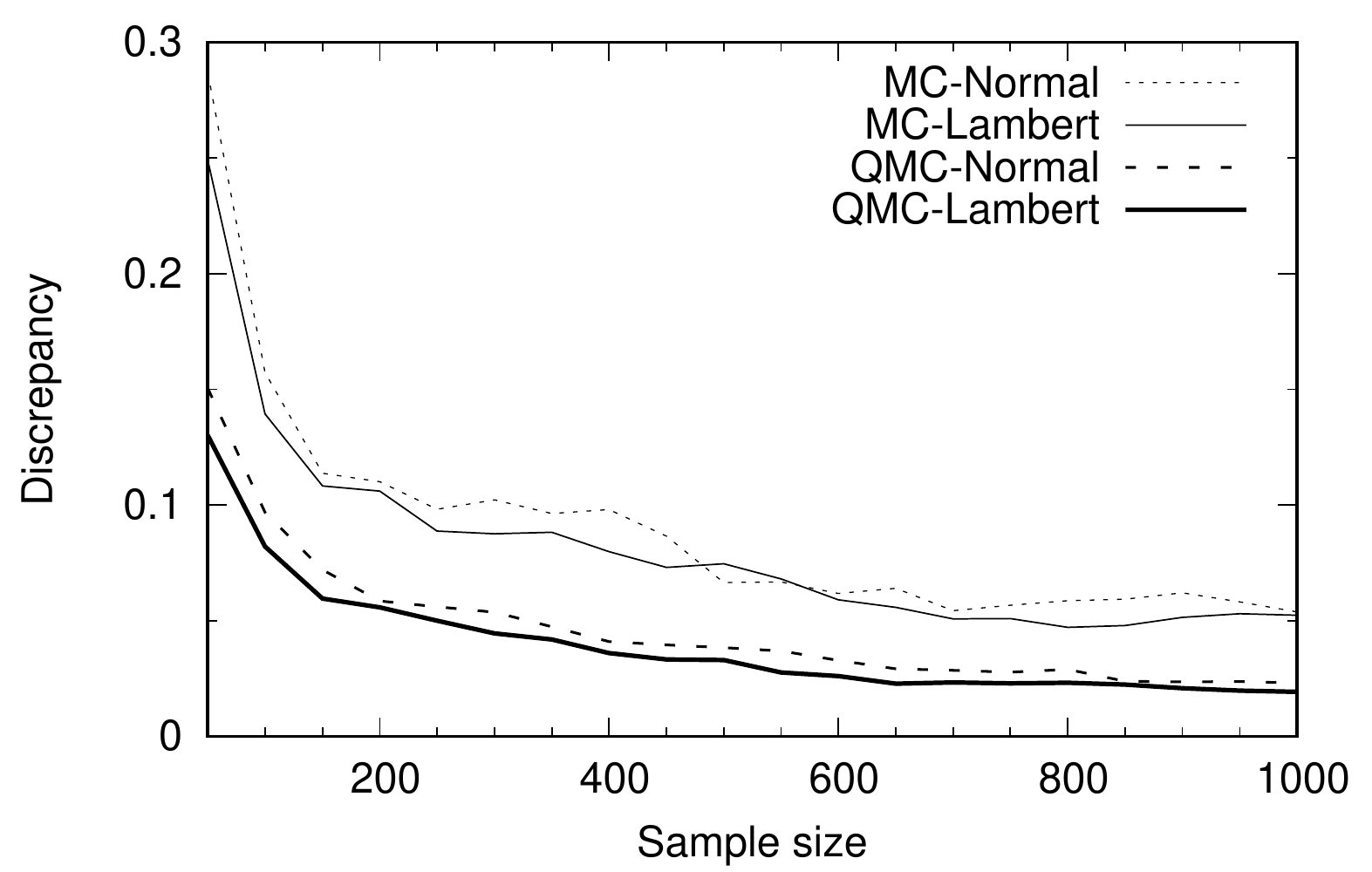}
\includegraphics[width=0.499\textwidth]{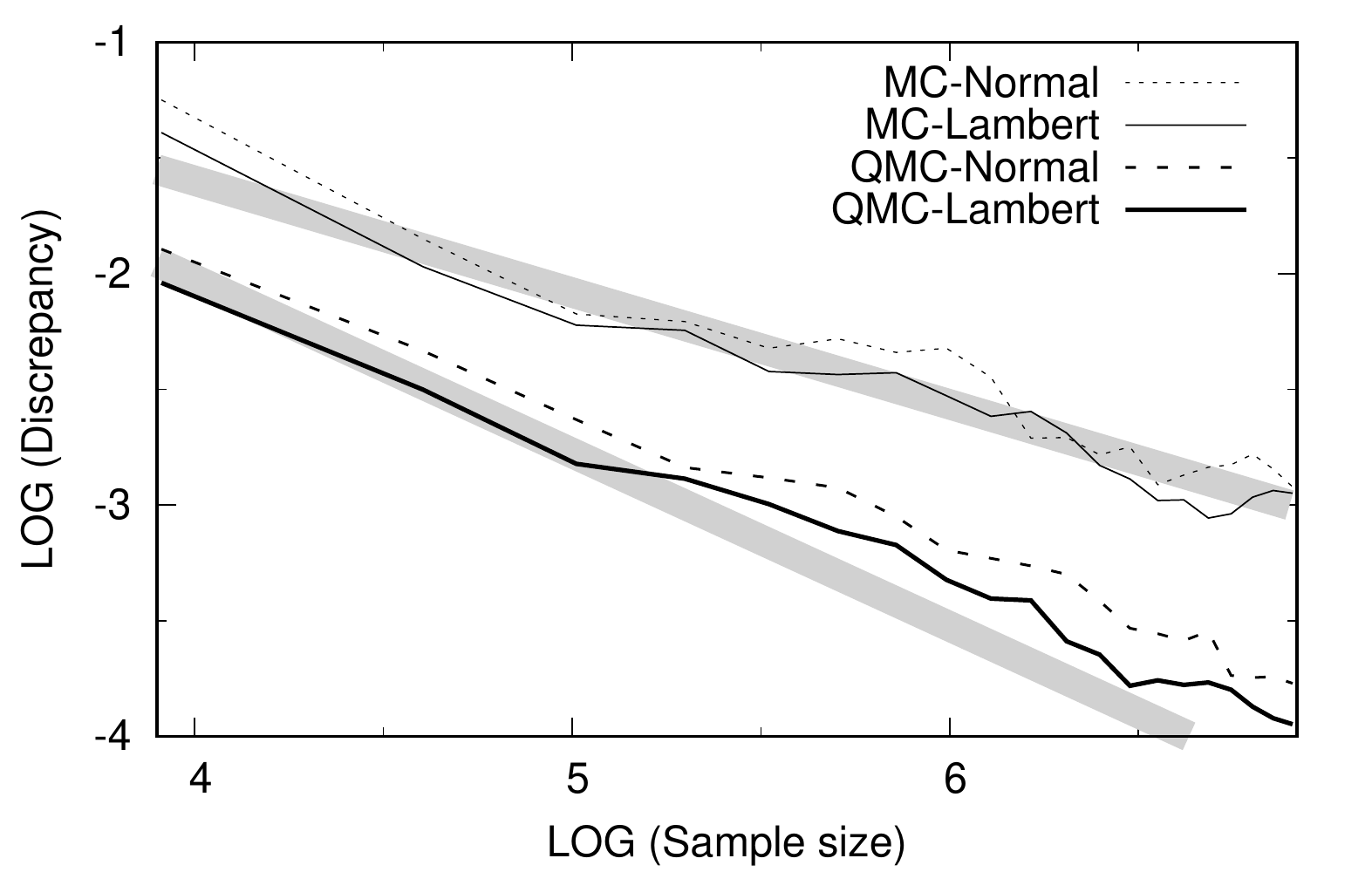}
\caption{Discrepancy as a function of sample size for 4 different sampling schemes in original form (left) and Log-Log Plot (right). For details see text.}
\label{decay}
\end{figure}

Figure \ref{decay} (left) shows the corresponding plots of the discrepancy as a function of the sample size 
(20 steps with an increment of sample size by 50 points at a time).
Not surprisingly, the QMC-based samples clearly outperform their MC counterparts. Moreover, in both classes, the Lambert transformation yields slightly better results than the normalization of Gaussians. The Log-Log plot (right) incorporates two gray strips with slopes identical to -1/2 (upper strip) and -3/4 (lower strip) with empirically shifted intercepts. It can be seen that the MC-based methods are closely tied with the expected decay rate of -1/2, whereas the QMC counterparts get a slope slightly above the optimal rate of -3/4.

\section*{Acknowledgement}

\noindent
The first author acknowledges support by DFG in the Collaborative Research
Centre CRC/Transregio 154, Project B04. The second author acknowledges
support by the FMJH Program Gaspard Monge in optimization and operations
research including support to this program by EDF. Moreover, the authors thank M. Messerschmid (Humboldt University Berlin) for discussion on the topic of this paper.

\bibliography{bibliography}
\bibliographystyle{plain}

\end{document}